\documentclass[a4paper, 10pt, journal]{IEEEtran}      % Use this line for a4 paper

\usepackage{color}
\usepackage{graphicx}
\usepackage{subcaption}
\usepackage{amsfonts,amssymb}
\usepackage{amsmath}
\usepackage{amsthm}
\usepackage{mathrsfs}
\usepackage{arydshln}
\usepackage{cases}
\usepackage{cite}
\usepackage{soul}

\usepackage{algorithm}
\usepackage{algorithmic}

\newcommand{\eio}{e^{i\omega}}

\newtheorem{thm}{Theorem}
\newtheorem{exm}{Example}
\newtheorem{lem}{Lemma}
\newtheorem{rem}{Remark}
\newtheorem{pro}{Proposition}

\newtheorem{defn}{Definition}
\newtheorem{coro}{Corollary}

\newtheorem{assu}{Assumption}

\newcommand{\vn}{\varnothing}

\begin{document}
	
	\title{\LARGE \bf
		Allocation of Excitation Signals for Generic Identifiability of Linear Dynamic Networks}

	\author{Xiaodong~Cheng,~\IEEEmembership{IEEE Member},
		Shengling Shi,~\IEEEmembership{IEEE Student Member,}\\
		and Paul M. J. Van den Hof,~\IEEEmembership{IEEE Fellow}% <-this % stops a space
		%\thanks{*This work was not supported by any organization}% <-this % stops a space
		\thanks{This work is supported by
			the European Research Council (ERC), Advanced Research Grant SYSDYNET, under the European Unions Horizon 2020 research and innovation programme (Grant Agreement No. 694504).}%
		\thanks{
			X. Cheng is with the Control Systems Group, Department of Electrical Engineering,
			Eindhoven University of Technology,
			5600 MB Eindhoven,
			The Netherlands. He is currently a research associate in the Department of Engineering at the University of Cambridge, CB2 1PZ, United Kingdom.
			{\tt\small xc336@cam.ac.uk}
			
			S. Shi and P. M. J. Van den Hof are with the Control Systems Group, Department of Electrical Engineering,
			Eindhoven University of Technology,
			5600 MB Eindhoven,
			The Netherlands.
			{\tt\small \{s.shi, p.m.j.vandenhof\}@tue.nl}}%
		%    \thanks{Yu Kawano is with Graduate School of Informatics, Kyoto University, Sakyo-ku, Kyoto 606-8501, Japan
		%        {\tt\small ykawano@i.kyoto-u.ac.jp}}%
		%    \thanks{This work of Y. Kawano was partly supported by JST CREST.}
	}
	
	\maketitle
	%\thispagestyle{empty}
	%\pagestyle{empty}

	%%%%%%%%%%%%%%%%%%%%%%%%%%%%%%%%%%%%%%%%%%%%%%%%%%%%%%%%%%%%%%%%%%%%%%%%%%%%%%%%
	%\renewcommand{\abstractname}{abstract}
	\begin{abstract}
		
		A recent research direction in data-driven modeling
		is the identification of dynamic networks, in which measured vertex signals are interconnected by dynamic edges represented by causal linear transfer functions.
%		 \blue{\st{Assuming that the topology of the underlying network is known, and all the vertex signals are measured, the major question addressed in this paper is where to allocate external excitation signals such that all the edges can be consistently estimated from measurement data.}}
The major question addressed in this paper is where to allocate external excitation signals such that
a network model set becomes generically identifiable when measuring all vertex signals. To tackle this synthesis problem, a novel graph structure, referred to as \textit{directed pseudotree}, is introduced, and the generic identifiability of a   network model set can be featured by a set of disjoint directed pseudotrees that cover all the parameterized edges of an \textit{extended graph}, which includes the correlation structure of the process noises.
		Thereby, an algorithmic procedure is devised, aiming to decompose the extended graph into a minimal number of disjoint pseudotrees, whose roots then provide the appropriate locations for excitation signals.
		Furthermore, the proposed approach can be adapted using the notion of \textit{anti-pseudotrees} to solve a dual problem, that is to select a minimal number of measurement signals for generic identifiability of the overall network, under the assumption that all the vertices are excited.
	\end{abstract}

	%\begin{IEEEkeywords}
	%    Structure preserving model reduction; Multi-agent systems; Network clustering; Large-scale system.
	%\end{IEEEkeywords}
	%%%%%%%%%%%%%%%%%%%%%%%%%%%%%%%%%%%%%%%%%%%%%%%%%%%%%%%%%%%%%%%%%%%%%%%%%%%%%%%%
	\section{Introduction}
	
	Dynamic networks adequately describe a wide class of complex engineering systems appearing in various applications, including  multi-robot coordination, power grids, and biochemical networks  see \cite{mesbahi2010graph} for an overview.
	The conventional system identification methods mainly focus on systems with relatively simple dynamical structures, e.g., single-input-single-output (SISO), multiple-input-multiple-output (MIMO), open-loop or closed-loop systems \cite{ljung1999system}. %,paul1995closedloop}.
%	\blue{\st{However, these classical data-driven techniques seem to be limited when encountering dynamic networks with complex interconnection structures.}}
	As control and design optimization for structured systems are resolved increasingly in a decentralized or distributed fashion, challenges arise in developing new data-driven modeling frameworks that address interconnection structures in network systems.
	
 	The interconnection structure of dynamic networks can not only capture the collective behavior of interacting dynamical subsystems but can also be used to represent causal dependencies among manifest signals \cite{chetty2015network}. Thereby, different representations of dynamic networks are considered. The first one focuses on interconnections of subsystems, see e.g., 
 	\cite{chetty2017necessary,suzuki2013node,haber2014subspace,henk2019systemtopology,Yu2018Subspace} and the references therein. 
	%    A series of methods, including node knockout \cite{nabi2012network,suzuki2013node}, subspace identification \cite{haber2014subspace}, synchronization \cite{wu2015identifying,henk2019topology}, Markov parameters \cite{henk2018undirected}.
	The second way is to consider signal structures. Specifically, the vertices in a network are interpreted as measured internal signals, and the directed edges represent transfer operators, referred to as \textit{modules}. Taking into account external noises and excitation signals, the identification of the modules in a network becomes a generalization of a closed-loop system identification problem \cite{paul2013netid}.

	With the latter description of dynamic networks, three research topics have been addressed. The first is to detect the topology of a network using measured internal signals, see e.g., \cite{materassi2010topology,materassi2012problem,hayden2016sparse,sanandaji2011acc,chiuso2012bayesian,shi2019bayesian,Zorzi&Chiuso:17}, where
	techniques, such as Wiener filters, compressed sensing or Bayesian approaches are taken to reconstruct the link structure among the process signals and obtain some sparse estimates.
	
	The second problem is to estimate a desired local module within a network. Various methods based on the \textit{prediction error method} can be found in e.g., \cite{dankers2016identification,ramaswamy2018local,gevers2015identification,Linder&Enqvist_ijc:17,gevers2018practical,Karthik2021TAC,shi2020excitation,shi2020single,shi2020generic}, which focus on the selection of predictor inputs: which signals are required to be measured such that we are able to consistently identify the dynamics of a particular module in the network?
	
	Relevant to the above question, the third problem, which is of particular interest in this paper, concerns the structural identifiability of a full dynamic network. Based on the results for deterministic network reconstruction problems in \cite{gonccalves2008LTInetworks,adebayo2012dynamical}, the concept of global network identifiability was introduced in an identification setting in \cite{Weerts&etal_IFAC:15,weerts2018identifiability}, as a property that reflects the ability to distinguish between network models in a parameterized model set on the basis of measurement data. 
%
%	In a second and related line of development network identifiability was defined as a property of single network, while, in a deterministic setting, introducing the concept of generic identifiability \cite{bazanella2017identifiability,hendrickx2018identifiability} referring to the situation that network models can be distinguished from {\it almost all} alternative models. The idea of introducing the aspect of genericity was then followed up by \cite{weerts2018single} who exploited this notion in the network identifiability setting of \cite{weerts2018identifiability} and by \cite{henk2018identifiability} who further explored path-based conditions for global identifiability.}
%	
%	
%	\cite{gonccalves2007dynamical,adebayo2012dynamical,gonccalves2008LTInetworks,weerts2018identifiability,hendrickx2018identifiability}. %{\blue The pioneering works in \cite{gonccalves2007dynamical,adebayo2012dynamical,gonccalves2008LTInetworks} define the \textit{dynamical %structure function} of a system, and the identifiability analysis considers a network reconstruction problem that finds%
%	 such a function consistent with the input-output transfer function of the system. }
%\blue{\st{Assuming that the topology of a network is known a priori, and a model set is considered in which all the models are associated %with the given topology. Then the}
% Network identifiability essentially 
 %
 In the literature, there are two classes of network identifiability, namely, \textit{global identifiability}  \cite{Weerts&etal_IFAC:15,weerts2018identifiability,henk2018identifiability} that requires models to be distinguishable from \textit{all} other models in the model set\footnote{There are actually two versions of global identifiability, reflecting whether either one particular model in the set can be distinguished or all models in the set \cite{weerts2018identifiability}.}, and \textit{generic identifiability}  \cite{bazanella2017identifiability,hendrickx2018identifiability,weerts2018single}, which means that models can be distinguished from \textit{almost all} models in the model set.
 Furthermore, 
The conditions for network identifiability have been analyzed within different settings. In e.g., \cite{bazanella2017identifiability,hendrickx2018identifiability,henk2018identifiability}, all vertices are excited by external excitation signals, while only a subset of vertices is measured. In contrast, the analysis in e.g., \cite{weerts2018identifiability,weerts2018single} assumes that all vertices are measured, while only a subset of vertices is excited. A recent contribution \cite{Bazanella&etal_CDC:19} also addresses the combined situation.

In all these settings, network identifiability is dependent on several structural properties of the model set, including the network topology, the modeled correlation structure of process noises, the presence and location of external excitation signals and the choices of measured vertex signals. Based on these properties, the existing results have provided both algebraic and graph-based analysis for network identifiability, that are typically formulated for each node separately and require a separate check of each and every node. However, none of them has referred to the synthesis problem, that is: where to
	allocate a limited number of excitation or measurement signals
%	\blue{\st{such that all the parameterized modules in the network can be consistently estimated via input-output data}
		so as to achieve network identifiability for the full network. Actually, such a problem has more realistic significance in the identification of dynamic networks, since it actually determines the cost of identification experiments in networks.
	This becomes the motivation of the current study.
% \blue{\st{Assuming a given topology of a network, we}
		We mainly focus on the situation that all the internal signals are measured, and we aim for a systematic scheme that allocates  the minimum number of excitation signals to achieve generic identifiability. To the best of our knowledge, such a synthesis problem has not been addressed in the literature so far.
	%Similar problems and their corresponding solutions can be found in the study of structural controllability for state space models \cite{olshevsky2015minimum,assadi2015complexity}, which is hard to be generalized for dynamic networks.
	
	In this paper, the main objective is to present a novel graph-theoretic approach to both  the analysis and synthesis of dynamic networks. Although \cite{hendrickx2018identifiability,weerts2018single} have provided attractive path-based conditions for checking the generic identifiability, the validation has to be carried out for each vertex, limiting the potential of these conditions for the use in the synthesis problem, particularly when large-scale or complex-structured networks are considered. In contrast to the path-based conditions, this paper introduces a novel graph structure, called \textit{directed pseudotrees}, and provides a different condition for guaranteeing generic identifiability of a full network using the concept of \textit{disjoint pseudotree covering}. More specifically, we define an \textit{extended graph}, which integrates the interconnection structure of the original network and the correlation structure of process noises. Then, the identifiability is characterized by a set of (edge) disjoint directed pseudotrees that cover all the parameterized edges of the extended graph, while each of the pseudotrees has a single external excitation.
	
	Based on this characterization, we find that the minimal number of excitation signals required for the identifiability is upper-bounded by the cardinality of the covering. Thereby, an effective heuristic algorithm is designed to decompose the extended graph into a minimal number of disjoint pseudotrees, whose roots, in fact, provide potential locations for excitation signals. The main ingredient of this algorithm is the concept of \textit{characteristic matrix}, which features all the pairs of mergeable pseudotrees in a covering. The graph merging steps are then completely carried out by using specific algebraic operations on the characteristic matrix. As a crucial follow-up step, we further check the necessity of stimulating one root of each pseudotree in the resulting covering. If it does not change the generic identifiability of the full network by excluding a pseudotree to have an excitation, we then reduce the required number of excitation signals.
	%     Furthermore, we also the proposed approach can be ed using the notion of \textit{anti-pseudotrees} to solve a dual problem, that is to select a minimal number of measurement signals for the generic identifiability of the overall network, under the assumption that all the vertices are excited.
	The current paper significantly improves the preliminary results in \cite{cheng2019CDC}, where the identifiability condition is only sufficient. Moreover, this paper considers a more general model setting, which allows for correlated noises and possible a priori known non-parameterized modules.

	The rest of this paper is organized as follows: In Section~\ref{sec:preliminaries}, we recapitulate some basic terminologies and notations in graph theory and provide the linear dynamic network model used in this paper. The definition of network identifiability is given in Section~\ref{sec:identifiability}, and  Section~\ref{sec:pseudotree} then defines a new graph structure, referred to as pseudotrees, and relevant concepts including disjoint pseudotrees and edge covering are introduced. In Section~\ref{sec:result}, we present a generic identifiability condition based on disjoint pseudotrees and then propose a pseudotree merging approach for the allocation of excitation signals in Section~\ref{sec:dual}. Finally, concluding remarks are made in Section~\ref{sec:conclusion}.
	
	\textit{Notation:}
	Denote $\mathbb{R}$ as the set of real numbers, and $\mathbb{R}(z)$ is the rational function field over $\mathbb{R}$ with variable $z$. $v_i$ denotes the $i$-th element of a vector $v$, and $A_{ij}$ denotes the $(i,j)$-th entry of a matrix $A$.
	The cardinality of a set $\mathcal{V}$ is given by $\lvert \mathcal{V} \rvert$.
	Let $\mathcal{G}$ be a directed graph, and we denote $V(\mathcal{G})$ and $E(\mathcal{G})$ as the vertex set and edge set of $\mathcal{G}$, respectively. The union of two graphs $\mathcal{G}_1$ and $\mathcal{G}_2$ is denoted by $\mathcal{G}:= \mathcal G_1 \cup \mathcal G_2$, where $V(\mathcal{G}) = V(\mathcal{G}_1) \cup V(\mathcal{G}_2) $ and $E(\mathcal{G}) = E(\mathcal{G}_1) \cup E(\mathcal{G}_2)$.
	
	%%%%%%%%%%%%%%%%%%%%%%%%%%%%%%%%%%%%%%%%%%%%%%%%%%%%%%%%%%%%%%%%%%%%%%%%%%%%%%%%
	
	\section{Preliminaries and Problem Setting}
	\label{sec:preliminaries}
	\subsection{Graph theory}
	\label{sec:preliminaries:graph}
	
	We provide necessary terminologies and concepts from graph theory and refer to \cite{mesbahi2010graph,godsil2013algebraic} for more details.
	The topology of a dynamic network is characterized by a graph $\mathcal{G}$ that consists of a finite and nonempty vertex set  $\mathcal{V}: = \{1, 2, ... , L\}$ and an edge set $\mathcal{E} \subseteq \mathcal{V} \times \mathcal{V}$.
	A directed graph is such that each element in $\mathcal{E}$ is an ordered pair of elements of $\mathcal{V}$. If $(i,j) \in \mathcal{E}$, we say that the edge is incident from vertex $i$ to vertex $j$, and the vertex $i$ is the \textit{in-neighbor} of $j$, and $j$ is the \textit{out-neighbor} of $i$. Let $\mathcal{N}_j^-$ and $\mathcal{N}_j^+$ be the sets that collect all the in-neighbors and out-neighbors of vertex $j$, respectively.
	
	A graph $\mathcal{G}$ is called \textit{simple}, if $\mathcal{G}$ does not contain self-loops (i.e., $\mathcal{E}$ does not contain any edge of the form $(i,i)$, $\forall~i\in \mathcal{V}$), and there exists only one directed edge from one vertex to each of its out-neighbors. In a simple graph, a directed \textit{path} connecting vertices
	$i_0$ and $i_n$ is a sequence of  edges of the form $(i_{k-1}, i_k)$, $k = 1, ..., n$, and every vertex appears at most once on the path. Two directed paths are \textit{vertex-disjoint} if they do not share any common vertex, including the start and the end vertices. In a simple directed graph $\mathcal{G}$, we denote $b_{\mathcal{U} \rightarrow \mathcal{Y}}$ as the maximum number of mutually vertex-disjoint
	paths from $\mathcal{U} \subseteq \mathcal{V}$ to $\mathcal{Y} \subseteq \mathcal{V}$. A directed simple graph $\mathcal{G}$ is \textit{connected} if the underlying undirected graph $\mathcal{G}_\mathrm{u}$ obtained by replacing all directed edges of $\mathcal{G}$ with undirected edges is connected, i.e., in $\mathcal{G}_\mathrm{u}$, there is an undirected path between any pair of vertices.
	
	In a simple connected graph $\mathcal{G}$, a source is a vertex without any in-neighbors, and likewise, a sink is a vertex without any out-neighbors. The sources and sinks of $\mathcal{G}$ are collected by 
	$
	{\mathcal S_\mathrm{ou}}(\mathcal{G}): = \left\{j \in V(\mathcal{G}) \mid |\mathcal{N}_j^-| = 0 \right\},
	$
	$  
	{\mathcal{S}_\mathrm{in}}(\mathcal{G}): = \left\{j \in V(\mathcal{G}) \mid |\mathcal{N}_j^+| = 0\right\}.
	$, respectively.

	\subsection{Dynamic Network Model}
	\label{sec:dynamodel}
	Consider a dynamic network whose topology is captured by a simple directed graph $\mathcal{G} = (\mathcal{V}, \mathcal{E})$ with vertex set
	$\mathcal V = \{1, 2, ... , L\}$ and edge set $\mathcal E \subseteq \mathcal V \times \mathcal V$. Following the basic setup of \cite{paul2013netid,weerts2018identifiability}, the dynamics of the $j$-th vertex in $\mathcal{G}$ is described by an internal variable $w_j(t) \in \mathbb{R}$ as
	\begin{equation} \label{eq:singlenode}
	w_j(t) = \sum_{l \in \mathcal{N}^-_j}^{} G_{jl}(q) w_l (t) + \sum_{k=1}^{K} R_{jk}(q)  r_k(t) + v_j(t),
	\end{equation}
	where $q^{-1}$ is the delay operator, i.e. $q^{-1} w_{j}(t) = w_j(t-1)$. $G_{jl}(q) \in \mathbb{R}(q)$ is referred to as a \textit{module} of the network, and $G_{jl}(q)$ is nonzero only if the edge $(l,j) \in \mathcal{E}$. Note that $G_{jj} = 0$, for all $j \in \mathcal{V}$, due to the simpleness of $\mathcal{G}$. The signals $r_k(t) \in \mathbb{R}$, with $k = 1, 2, ..., K$, are the external excitations that can directly be manipulated by users.
	Denote $\mathcal{R} \subseteq \mathcal{V}$, with $|\mathcal{R}| = K$, as the set of vertices that are affected by the external excitation signals, thereby $R_{jk}(q) \in \mathbb{R}(q)$ is nonzero if the vertex $j \in \mathcal{R}$ is excited by $r_k(t)$, and $R_{jk}(q) = 0$ otherwise. Moreover, $v_j(t) \in \mathbb{R}$ is the unmeasured disturbance injected into the $j$-th node.
	
	A compact form for expressing the dynamics of the network is obtained as
	\begin{equation} \label{eq:net}
	w(t) = G(q) w(t) +  R(q)  r(t) + v(t),
	\end{equation}
	where $G(q)$, $R(q)$ are the transfer matrices that collect $G_{jl}(q)$ and $R_{jk}(q)$ in \eqref{eq:singlenode} as their corresponding entries, respectively. $w(t): = \begin{bmatrix}
	w_1(t) & w_2(t)& ... & w_L(t)
	\end{bmatrix}^\top$,
	$r(t) : = \begin{bmatrix} r_1(t) & r_2(t) & ... & r_K(t)\end{bmatrix}^\top$,
	and $v(t) : = \begin{bmatrix} v_1(t)& v_2(t) & ... & v_L(t)\end{bmatrix}^\top$. For the identifiability analysis in this paper, the signals $w(t)$ and $r(t)$ are assumed to be known.
	
	\begin{assu} \label{Assum}
		Throughout the paper, we consider a dynamic network \eqref{eq:net} with the following properties.
		\begin{enumerate}
			
			\item The network \eqref{eq:net} is \textit{well-posed} and stable, i.e., $(I - G(q))^{-1}$ is proper
			and stable.
			
			\item All the entries of $G(q)$ and $R(q)$ are proper and stable transfer operators, and each row of $R(q)$ contains only one nonzero entry, i.e., each vertex in $\mathcal{R}$ is influenced by a single excitation signal.
			\item $v(t)$ is modeled as a stationary stochastic process with a rational spectral density:
			\begin{equation} \label{eq:noisemodel}
			v(t) = H(q)e(t),
			\end{equation}
			where $e(t):= \begin{bmatrix}e_1(t) & e_2(t) &... & e_p(t)\end{bmatrix}^\top$ is a
			white noise process, with dimension $p \leq L$ and the covariance matrix $\Lambda > 0 $.
			In the case of $p = L$, $H(q)$ is a proper rational transfer matrix which is monic, stable and minimum-phase. For the situation $p < L$, i.e., rank-reduced noises, 
 {$H(q)$ is structured as $H(q) = \begin{bmatrix} H_a \\ H_b \end{bmatrix}$, with $H_a$ square, proper, monic, stable and minimum phase}, see \cite{weerts2018identifiability} for more details. \hfill $\Box$
		\end{enumerate}
	\end{assu}
	
The above are standard assumptions made for dynamic networks to ensure the properness and stability of the mapping from $r(t)$ to $w(t)$ and of the noise model, which are essential for the identifiability analysis, see  \cite{weerts2018identifiability,weerts2018single} for more details. 

%To tackle the above synthesis problem, the definition of generic identifiability is first studied, which is then characterized by the so-called \textit{disjoint pseudotree covering}. Thereby, a novel scheme is proposed, aiming to decompose a given directed graph $\mathcal{G}$ into a minimal number of edge-disjoint pseudotrees, which provides a solution of selecting excitation vertices for generic identifiability.

\subsection{Generic Identifiability}
\label{sec:identifiability}
	
	In order to define network identifiability, a network model and a network model
	set are specified. Consider a dynamic network as in \eqref{eq:net} of $L$ internal signals, $K$ external excitation signals, and a noise process of rank $p \leq L$. Following \cite{weerts2018identifiability}, a network model is defined by the quadruple
	\begin{equation} \label{eq:networkmodel}
		M = (G, R, H, \Lambda),
	\end{equation}
	where $G \in \mathbb{R}(z)^{L \times L}$, $R \in \mathbb{R}(z)^{L \times K}$, $H \in \mathbb{R}(z)^{L \times p}$ are \textit{proper} transfer matrices satisfying the properties in Assumption~\ref{Assum}, and $\Lambda \in \mathbb{R}^{p \times p}$ is the \textit{positive definite} noise covariance matrix.
	We then denote a set of parameterized matrix-valued functions
	\begin{equation} \label{eq:modset}
	\mathcal{M}: = \{M(q,\theta) = (G(q,\theta),R(q),H(q,\theta), \Lambda(\theta)), \theta \in \Theta\}
	\end{equation}
	as the network model
	set with all network models $M(\theta)$
	described in \eqref{eq:networkmodel}. The network model set $\mathcal{M}$ represents prior knowledge of the dynamic network including the topology, non-parameterized modules, presence, disturbance correlation, and locations of external signals. All the entries of $R(q)$ are known and thus non-parameterized.
%	\red{\st{and the structure of $G$ is consistent with the topology of the network, which is known \textit{a priori}, i.e., $G_{ji} = 0$, if $(j,i) \notin \mathcal{E}$.}}
	Note that the variable $\theta \in \Theta$ in \eqref{eq:modset} is only used for formalizing a set of models, while the properties of the mapping from $\theta$ to network models will not be addressed.

Denote the transfer matrix
\begin{equation} \label{eq:T}
T(q, \theta) = \begin{bmatrix}
T_{wr}(q, \theta) & T_{we}(q, \theta)
\end{bmatrix},
\end{equation}
where $T_{wr}(q, \theta): = (I - G(q,\theta))^{-1} 
R$ and  $T_{we}(q, \theta): = (I - G(q,\theta))^{-1} H(q,\theta)$, and we denote the signal
$\tilde v(t,\theta)$ as the disturbance signal with power spectrum $\Phi_{\bar v}(\omega,\theta) = T_{we}(\eio,\theta)\Lambda(\theta)T_{we}(\eio,\theta)^{-\star}$. In our identification setting, $w(t)$ and $r(t)$ are the measurement data, from which we can uniquely identify the transfer matrix $T_{wr}$ and the power spectrum   $\Phi_{\bar{v}}$, provided that we have sufficiently excitating signals $r$. Then, the concept of identifiability specifies whether there is a unique representation of a network model in the model set $\mathcal{M}$ that matches the objects $T_{wr}$ and $\Phi_{\tilde v}$. In the next definition we extend the formulation of global network identifiability as introduced in \cite{weerts2018identifiability} with the principle of genericity that was introduced in \cite{bazanella2017identifiability,hendrickx2018identifiability} for generic identifiability, but applied to a slightly different notion of identifiability.
\begin{defn}[Network identifiability]
		 \label{def:netid}
	Consider a network model set $\mathcal{M}$, and a model $M(q,\theta_0) \in \mathcal{M}$  for which we consider the following implication:
\begin{align}\label{eq:implication0}
	\left.\begin{aligned}
		T_{wr}(q, \theta_0) &= T_{wr}(q, \theta_1) \\
		\Phi_{\bar{v}}(\omega, \theta_0) &= \Phi_{\bar{v}}(\omega, \theta_1) 
	\end{aligned}\right\rbrace 
	\Rightarrow
  M(q,\theta_1) = M(q,\theta_0),  
	\end{align}
for all  $\theta_1 \in\Theta$.
Then $\mathcal{M}$ is 
\begin{itemize}
    \item[a.] globally network identifiable from $(r,w)$ if implication \eqref{eq:implication0} holds for all  $\theta_0 \in \Theta$;
    \item[b.] generically network identifiable from $(r,w)$ if implication \eqref{eq:implication0} holds for almost all\footnote{ ``Almost all'' refers to the exclusion of parameters that are in a subset of the finite set $\Theta$ with Lebesgue measure $0$.  When the parameter space $\Theta$ is of infinite dimension, we consider the concept of generic properties in a topological space \cite{tchon1983generic} applied to the space of models instead, from which a more rigorous definition of generic identifiability is introduced, see \cite{shi2020genericity} for the details.  }  $\theta_0 \in \Theta$. \hfill $\Box$
\end{itemize}
\end{defn}

In order to support the analysis and verification of network identifiability, we add the following step, that further simplifies the implication \eqref{eq:implication0}.  
\begin{lem}\cite{weerts2018identifiability}\label{lem1}
If model set $\mathcal{M}$ satisfies the condition that
 \begin{itemize}
 \item either all modules $G(q,\theta)$ are parameterized to be strictly proper, or
 \item the parameterized network model does not contain any algebraic loops\footnote{There exists an algebraic loop around node $w_{n_1}$ if there exists a sequence of integers $n_1,... n_k$ such that $G^{\infty}_{n_1n_2}G^{\infty}_{n_2n_3}... G^{\infty}_{n_kn_1} \neq 0$, with $G^{\infty}_{n_1n_2}:=\lim_{z\rightarrow\infty} G_{n_1n_2}(z)$.}, and  $H^{\infty}(\theta)\Lambda(\theta)H^{\infty}(\theta)^T$  is diagonal for all $\theta\in\Theta$,  with $H^\infty(\theta):=\lim_{z\rightarrow\infty}H(z,\theta)$, 
\end{itemize}
then implication \eqref{eq:implication0} can equivalently be formulated as 
\begin{align}\label{eq:implication}
	T(q, \theta_1) = T(q, \theta_0)
	\Rightarrow \left\{ \begin{array}{c}
	G(q,\theta_1) = G(q,\theta_0), \\ H(q,\theta_1)=H(q,\theta_0), \\ \end{array} \right.  
	\end{align} 
for all $\theta_1 \in\Theta$. 
\end{lem}

The basic step that is made in Lemma \ref{lem1} is to formulate conditions under which the transfer function $T_{we}$ can be uniquely recovered from the spectrum $\Phi_{\bar v}$, and thus the full matrix $T$ in \eqref{eq:T} can be obtained from measurement data $w(t)$ and $r(t)$.
Throughout this paper, we will assume that the considered model sets $\mathcal{M}$ will satisfy the conditions of Lemma~\ref{lem1} and so that we can use implication \eqref{eq:implication} for verifying network identifiabiltiy according to Definition \ref{def:netid}.

In the next step, implication (\ref{eq:implication}) is reformulated in terms of a condition on a particular matrix rank. For this step we need the following assumption that originates from \cite{weerts2018identifiability}.

\begin{assu} \label{ass2}
Consider the following two conditions on network model set $\mathcal{M}$ in \eqref{eq:modset}:
\begin{enumerate}
		\item[a)] Every parameterized entry in $\{G(q,\theta), H(q,\theta)\}$
		covers all proper rational transfer functions\footnote{within the constraints of the conditions of Lemma \ref{lem1}.};
%		{\red \footnote{{\red Alternative and improved formulation: each of the entries of the rational transfer functions $\{G(q,\theta), H(q,\theta)\}$ covers all rational models of any finite order subject to the conditions of Assumption 1.}}}
		\item[b)] All parameterized transfer functions $\{G(q,\theta), H(q,\theta)\}$ are
		parameterized independently. \hfill $\Box$
	\end{enumerate}
\end{assu}

In order to formulate the rank condition for satisfying implication (\ref{eq:implication}), we denote two important sets of signals:
\begin{align*}
&\mathcal{P}_j :=  \{ i\in \mathcal{N}_j^- \subset \mathcal{V}\ |\ G_{ji}(\theta)\ \mbox{is parameterized in } \mathcal{M}\}, \\
&\mathcal{U}_j :=  \mathcal{R} \cup \{e_{\ell}\ | H_{j\ell}(q) \mbox{ is non-parameterized in } \mathcal{M}\},
\end{align*}
and we define the transfer matrix $\breve T_j(\theta)$ as the transfer matrix from $\mathcal{U}_j \rightarrow \mathcal{P}_j$ for models in the model set $\mathcal{M}$.
%
%\begin{lem}[\cite{weerts2018identifiability}] \label{lem2}
%Consider a parameterized model set $\mathcal{M}$. If for each $j = 1,\ldots, L$, it holds that
%\[ \breve T_j(\theta_0)\ \  \mbox{has full row rank} \]
%then implication (\ref{eq:implication}) holds true.\\
%
%If $\mathcal{M}$ satisfies Assumption \ref{ass2} then the condition is also necessary. \hfill $\Box$
%\end{lem}
%
Now, in line with the step made in \cite{weerts2018single}, based on the introduction of genericity in the concept of identifiability according to \cite{bazanella2017identifiability,hendrickx2018identifiability}, we can formulate the following result for generic identifiability of $\mathcal{M}$:
\begin{pro} \label{pro1}
Let model set $\mathcal{M}$ satisfy the conditions of Lemma \ref{lem1}. If for each $j = 1,\ldots, L$ it holds that
\[ \breve T_j(\theta_0)\ \  \mbox{has full row rank for almost all } \theta_0\in\Theta\]
then $\mathcal{M}$ is generically identifiable from $(r,w)$. 
If $\mathcal{M}$ satisfies Assumption \ref{ass2} then the condition is also necessary. 
\end{pro} 
%\begin{proof}
%This is a direct result of Theorem 2 in \cite{weerts2018identifiability} where it is shown that under the given conditions, implication (\ref{eq:implication}) is equivalent to a full row rank property of $\breve T_j(\theta_0)$. Application of this result to the generic case then directly follows from the definition of generic identifiability in Definition \ref{def:netid} part b.
%\end{proof}
This is a direct result of Theorem 2 in \cite{weerts2018identifiability} and the definition of generic identifiability in Definition~\ref{def:netid} part b.

Based on the model setting in Section~\ref{sec:dynamodel}, this paper mainly addresses a synthesis problem in dynamic networks to achieve generic identifiability. Specifically, we are interested in allocating a minimal number of external excitation signals, i.e., find the set $\mathcal{R}$ of minimal cardinality, such that
%the all the parameterized transfer functions in $G(q)$ can be consistently identified from the external %excitation signals $r(t)$ and the measurement data $w(t)$.
network models in a model set can be distinguished on the basis of the measurement data $w(t)$ and the presence and location of external excitation signals $r(t)$ and noise disturbances $v(t)$.

\section{Generic Identifiability Based on Extended Graphs}

In this section, we introduce the concept of extended graphs for dynamic networks. An extended graph, which incorporates the underlying graph of a network and its structure of noise correlation, then leads to a path-based condition for checking generic identifiability.

The condition in Proposition~\ref{pro1} reflects for every vertex in the network, the generic (row) rank of a rational transfer matrix between a set of external signals (measured excitation signals and unmeasured stochastic disturbance signals) and a set of internal vertex signals in the network. In an important theorem of Van der Woude \cite{vanderWoude1991graph}, a connection has been made between the generic rank of a dynamic transfer matrix and path-based conditions applied to the graph of the network. This connection has been  exploited in \cite{bazanella2017identifiability,hendrickx2018identifiability} to establish path-based conditions for the generic rank of a dynamic transfer matrix in the setting that all the vertices of the dynamic network are excited by sufficiently rich external signals. Additionally, the existing path-based conditions for generic network identifiability require all the nonzero transfers in the network matrix $G(q)$ to be parameterized independently. For formulating path-based conditions for the considered situation in this paper, including disturbance inputs and noise models, we first impose an additional assumption:
\begin{assu}\label{ass3}
	In model set $\mathcal{M}$, all the nonzero entries in $G(q,\theta)$ are parameterized, and
	{each row and column of $H(q,\theta)$ contains either a single nonzero (parameterized or nonparameterized) entry or only multiple nonzero parameterized entries.}
	% $H(q,\theta)$ is a known selection matrix {each row of $H(q,\theta)$ contains only one nonzero entry, or has also all nonzero entries parameterized.
\end{assu}
This assumption on $H$ allows a $v$ signal being modeled as a white noise or multiple $v$ signals having correlations that are parameterized.  Furthermore, we define an auxiliary notion related to the graph of the network, in particular for the situation of having external disturbance signals incorporated.

%\red{THERE IS AN ISSUE IN THE SEQUEL. WE HAVE TO MORE CAREFULLY DISTINGUISH BETWEEN A NETWORK AND ITS GRAPH, OR THE GRAPH OF A parameterized MODEL. NOW YOU ARE USING THE SAME NOTATION $\mathcal{G}$ FOR THIS. I WOULD SUGGEST TO FORMULATE DEFINITION 2 IN TERMS OF AN EXTENDED GRAPH AND NOT AN extended network. THAN THIS EXTENDED GRAPH IS USED TO REPRESENT THE parameterized MODEL SET, WHICH FITS WITH THE USE OF A parameterized NOISE MODEL IN (13). PLEASE CHECK THE FULL PAPER ON CONSISTENCY IN THE USE OF THESE CONCEPTS. AT SOME POINT I STOPPED MAKING CHANGES CONCERNING THIS.}

	\begin{defn}[Extended graphs]\label{def:extendednet}
		Consider a dynamic network \eqref{eq:net} with the noise model \eqref{eq:noisemodel}. Let $\mathcal{G}$ be its underlying graph. An extended graph $\widehat{\mathcal{G}}$ of  the parameterized part of $\mathcal{M}$ is defined  by $V(\widehat{\mathcal{G}}) = V(\mathcal{G}) \cup \widehat{V}$, and $E(\widehat{\mathcal{G}}) = E(\mathcal{G}) \cup \widehat{E}$, where
		\begin{align*}
		\widehat{V}: &= \{L+1,L+2,...,L+p-p_0\}, 
		\\	
		\widehat{E}: &= \{(i,j) \mid j \in \widehat{V}, i\in \mathcal{V}, H_{i, j-L}(q,\theta)~\text{is parameterized} \}, 
		\end{align*}
		with $p_0$ the number of nonparameterized columns in $H(q,\theta)$.
	\end{defn}
Note that the extended graph $\widehat{\mathcal{G}}$ in Definition~\ref{def:extendednet} only captures the nonzero parameterized transfers in $G$ and $H$. The set $\widehat{V}$ collects additional vertices associated with the noises signals $e(t)$, from which there are parameterized mappings to the internal signals of the network \eqref{eq:net}. These parameterized mappings are then indicated by the edges in $\widehat{E}$.
	Thus, the extended graph $\widehat{\mathcal{G}}$ integrates the structure of the original graph  $\mathcal{G}$ and the correlation structure of the process noises simultaneously. Denote $\mathcal{U}$ as the set of \textit{stimulated vertices} in $\widehat{\mathcal{G}}$,  which are excited by the external signals $\mathcal{R} \cup \{ e_1, e_2, ... , e_p \}$,  and let $\widehat{\mathcal{P}}_j$ be the  set of in-neighbors of vertex $j$ in the extended graph $\widehat{\mathcal{G}}$.
In the following, we use  the extended graph of the network \eqref{eq:net} to characterize generic identifiability. 
	\begin{lem}\label{lem:genericid}
		Given a network model set $\mathcal{M}$ that satisfies the conditions in Lemma~\ref{lem1} and Assumptions~\ref{ass2} and \ref{ass3}. Then $\mathcal{M}$ is generically identifiable from $(r, w)$ if and only if in its extended graph $\widehat{\mathcal{G}}$,
		\begin{equation} \label{eq:maxpath}
		b_{\mathcal{U} \rightarrow \widehat{\mathcal{P}}_j} = |\widehat{\mathcal{P}}_j|
		\end{equation}
		holds for all $j \in V(\mathcal{G})$, where $b_{\mathcal{U} \rightarrow \widehat{\mathcal{P}}_j}$ is the maximal number of vertex-disjoint paths from  $\mathcal{U}$ to $\widehat{\mathcal{P}}_j$.
	\end{lem}
	\begin{proof}
For the situation of a dynamic network without disturbance signals, it has been shown in Proposition V.1 of \cite{hendrickx2018identifiability} that there is an equivalence between the generic row rank of the matrix transfer function $\breve T_j(\theta_0)$ and $b_{\mathcal{U}_j\rightarrow \mathcal{P}_j}$ in the graph that is related to the parameterized model set $\mathcal{M}$. For this equivalence it is required that all nonzero entries in the transfer function matrix are parameterized independently, relating back to the original system theoretic result of \cite{vanderWoude1991graph}, and that all modules are parameterized without a restriction on the model order, as formulated in  Assumption~\ref{ass2}. In \cite{hendrickx2018identifiability} this latter condition has been formulated, in a slightly different setting, as considering {\it any} rational transfer matrix parametrization consistent with the directed graph. If disturbance signals are included, we need to show that the same properties hold when using the extended graph. 
With Assumption~\ref{ass3}, the noise model in \eqref{eq:noisemodel} can be reformed as 
$v(t) = H_{\theta}(q, \theta) e_{\theta}(t) + H_f(q) e_f(t)$, 
where $e_{\theta}(t) \in \mathbb{R}^{p - p_0}$, $ e_f(t) \in \mathbb{R}^{p_0}$, and  all the nonzero entries of $H_{\theta}(q, \theta)$ are parameterized, while those of $H_f(q)$ are nonparameterized. 
%{\red \st{Thus, the white noises {\red $e_f(t)$} have the same functionality as $r(t)$ in the context of identifiability}}.
  Then,  the network equation \eqref{eq:net} can simply be rewritten as
\begin{equation}\label{eq:extg}
\underbrace{\begin{bmatrix} 
	w \\ w_e\end{bmatrix}}_{w'} 
= 
\underbrace{\begin{bmatrix} 
	G(q,\theta) & H_{\theta}(q, \theta) \\
	 0 & 0 
	\end{bmatrix}}_{G_{\mathrm{ext}}}
\underbrace{\begin{bmatrix} 
	w \\ w_e\end{bmatrix}}_{w'} 
+ \underbrace{\begin{bmatrix} 
	R(q) r +   H_f(q) e_f \\ e_{\theta} 
	\end{bmatrix}
}_{u}
\end{equation}
where $w_e = e_{\theta}$, and $G_{\mathrm{ext}}$ now reflects the network matrix of the extended network, in which all the nonzero entries are parameterized. Full rank properties of mappings from signals in $u$ to signals in $w'$ can now be derived using path-based conditions of the graph related to $G_{\mathrm{ext}}$, just like the results that have been derived in \cite{hendrickx2018identifiability}.
% In this setting, all non-zero entries of $H$ would also need to be parameterized. If $H$ is a selection matrix and non-parameterized, it can be moved from $G_{ext}$ to $u$ and the path-based analysis still applies. 
% In this setting, all non-zero entries of $H$ would also need to be parameterized. If $H$ is a selection matrix and non-parameterized, it can be moved from $G_{ext}$ to $u$ and the path-based analysis still applies. 
This proves the condition $b_{\mathcal{U}_j\rightarrow \mathcal{P}_j} = |\mathcal{P}_j|$. 
Since in the extended graph  
Note that in the extended graph, $\mathcal{U}\backslash\mathcal{U}_j$ is the set of $e$ signals which have  parameterized edges incident to node $j$, where this set coincides with $\widehat{\mathcal{P}}_j\backslash \mathcal{P}_j$. Thus, it is verified that $b_{\mathcal{U}\backslash\mathcal{U}_j \rightarrow \widehat{\mathcal{P}}_j\backslash \mathcal{P}_j} = |\widehat{\mathcal{P}}_j|-|\mathcal{P}_j|$ where the corresponding vertex disjoint paths are vertex disjoint with the vertex disjoint paths from $\mathcal{U}_j$ to $\mathcal{P}_j$. 
This proves that the condition $b_{\mathcal{U}_j\rightarrow \mathcal{P}_j} = |\mathcal{P}_j|$ is equivalent to \eqref{eq:maxpath}.   
	\end{proof}

%{\red I do not see the point in the next reasoning and find it confusing. From my perspective the main message should be that $r$ and $e$ play a similar role when it comes down to the verification of the rank condition. That's it, isn't it? The message that $e$ is always connected to a parameterized H is in conflict with the option of having $H$ as a selection matrix.}\\
%{\green Originally, I treat the noise model as the unknown, that is all the entries in H are parameterized. Now since we also take into account known entries in H, then we can change the definition of extended graphs and exclude the effect of the known entries. Specifically, in (11) and (12), we change \\
%1) $\widehat{V}: = \{L+1,L+2,...,L+p\}$ to $\widehat{V}: = \{L+1,L+2,...,L+\hat{p}\}$, where $\hat{p} \leq p$, and $p - \hat{p}$ is the number of rows in H, whose entries are all known.\\
%2) $H_{i, j-L}(q,\theta) \ne 0$ to $H_{i, j-L}(q,\theta)$ is parameterized. \\
%Then, it guarantees all the outgoing edges from the extended nodes (i.e., the $e$ nodes) are parameterized. }

Note that excitation signals $r(t)$ and noises $e(t)$ contribute differently to the generic identifiability of the model set $\mathcal{M}$, and in the construction  of  extended graphs in Definition~\ref{def:extendednet}, we interpret all the parameterized entries in $H(q)$ as edges in $\widehat{E}$. In this way, the notion  of  extended  graphs $\widehat{\mathcal{G}}$ unifies the roles of external signals $r(t)$ and $e(t)$, while the only difference is  that $e(t)$ are always connected to a subset of vertex signals $w(t)$ via parameterized edges in $\widehat{\mathcal{G}}$.
Therefore, a concise characterization of generic identifiability can be provided in Lemma~\ref{lem:genericid} for dynamic networks with correlated  noises, whose correlation structure is captured in the corresponding extended graph as well.  The condition in Proposition~\ref{pro1} can now be checked using only one equality \eqref{eq:maxpath}, and moreover this checking is based on the vertex-disjoint paths from a common set $\mathcal{U}$ of stimulated vertices to all the in-neighbors of different vertex in $\widehat{\mathcal G}$.

	In the following example, we  demonstrate how the extended graph $\widehat{\mathcal{G}}$ is constructed and how it is used to check the generic identifiability of $\mathcal{M}$.
	\begin{exm} \label{exam1}
				\begin{figure}[!tp]\centering
			\begin{minipage}[t]{\linewidth}
				\includegraphics[scale=.3]{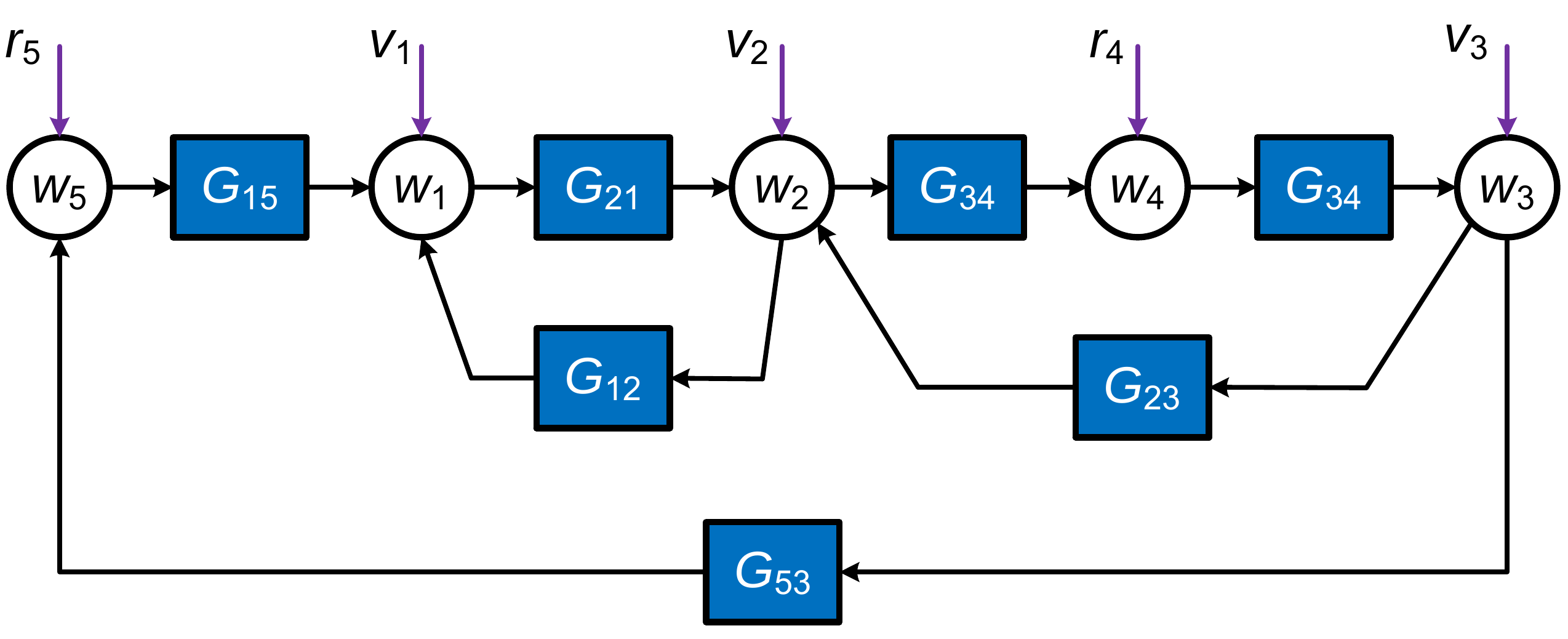}
				\centering
				\subcaption{}
				\label{fig:5nodenet}
			\end{minipage}
			\begin{minipage}[t]{\linewidth}
				\centering
				\includegraphics[scale=.33]{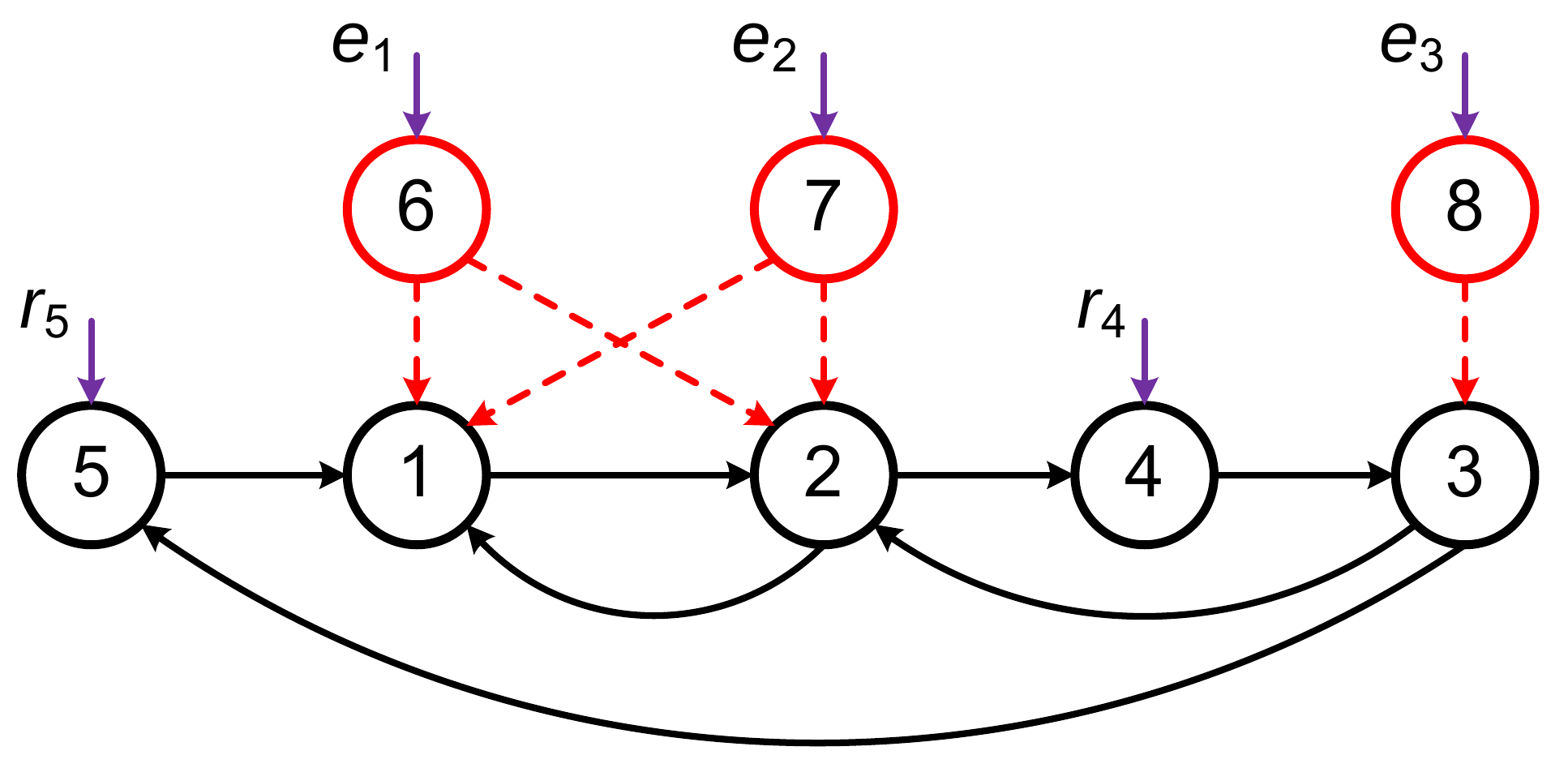}
				\subcaption{}
				\label{fig:ex5nodenet}
			\end{minipage}%
			\caption{Illustration of the extended graph of a given dynamic network. (a) The original dynamic network $\mathcal{G}$, in which $v_1$ and $v_2$ are correlated process noises; (b) The extended graph $\widehat{\mathcal{G}}$, in where the dashed edges are additional parameterized edges. }
		\end{figure}
	
		Consider a dynamic network shown in Fig.~\ref{fig:5nodenet}, where $v_1(t)$ and $v_2(t)$ are correlated such that
		\begin{equation*}
		\setlength{\dashlinegap}{2pt}
		\left[
		\begin{array}{c:c}
		H(\theta) & R
		\end{array}
		\right]
		= \left[
		\begin{array}{ccc:cc}
		H_{11}(\theta) & H_{12}(\theta) & 0 & 0 & 0\\
		H_{21}(\theta) & H_{22}(\theta) & 0 & 0 & 0\\
		0 & 0 & H_{33}(\theta) &  0 & 0\\
		0 & 0 & 0 &  1 & 0\\
		0 & 0 & 0 &  0 & 1
		\end{array}
		\right].
		\end{equation*}
%		in the network model set $\mathcal{M}$. 
		By Definition~\ref{def:extendednet}, the extended graph $\widehat{\mathcal{G}}$ is defined and shown in Fig.~\ref{fig:ex5nodenet}, where $\widehat{V}=\{6,7,8\}$ is the set of additional vertices added to $\mathcal{G}$, and $\widehat{E}=\{(1,6),(2,6),(1,7),(2,7),(3,8)\}$ are generated based on $H(\theta)$, indicating the edges directed from $\widehat{V}$ to a subset in $\mathcal{V}$.
		
		We now make use of the extended graph in Fig.~\ref{fig:ex5nodenet} to check the generic identifiability of the dynamic network set $\mathcal{M}$. In $\widehat{\mathcal{G}}$, the set of stimulated vertices is $\mathcal{U}: = \{4,5,6,7,8\}$, and the in-neighbors of vertex $1$ are collected in $\widehat{\mathcal{P}}_1 = \{2, 5, 6, 7\}$. Clearly,
		there exist $4$ vertex-disjoint
		paths from $\mathcal{U}$ to $\widehat{\mathcal{P}}_1$, namely, the condition \eqref{eq:maxpath} holds for $j = 1$.
		We continue to verify \eqref{eq:maxpath} for the other vertices $j \in \mathcal{V} = \{1,2,3,4,5\}$ and find that the maximal number of vertex-disjoint paths in $\widehat{\mathcal{G}}$ from $\mathcal{U}$ to $\widehat{\mathcal{P}}_j$ is always equal to $|\widehat{\mathcal{P}}_j|$. Therefore, the network model set $\mathcal{M}$ is generically identifiable.

	\end{exm}

	For the synthesis problem studied in this paper, the condition in Lemma~\ref{lem:genericid} is still not convenient enough to use, as it requires to check the equation \eqref{eq:maxpath} vertex by vertex. Thus, we will introduce in Section~\ref{sec:pseudotree} a novel graph concept, called \textit{pseudotrees}, and relevant results on disjoint pseudotree covering. Then in Section~\ref{sec:result}, a new characterization of generic identifiability will be presented based on disjoint pseudotrees, which further leads to an excitation signal allocation approach for generic identifiability.

	\section{Disjoint Pseudotree Covering}
	\label{sec:pseudotree}
	We make the result of this section self-contained and
	independent of the signal allocation problem of dynamic networks.
	In this section, a novel graph concept, called \textit{directed pseudotree}, is introduced.
	
	\begin{defn}[Directed pseudotrees] \label{def:pseudotree}
		A connected simple directed  graph $\mathcal{T}$, with $|V(\mathcal{T})| \geq 2$, is called a (directed) \textbf{pseudotree} if   $|\mathcal{N}_i^-| \leq 1$, for all $i \in V(\mathcal{T})$.
	\end{defn}
	
	The above concept of pseudotrees is an extension of its definition in the undirected case, in which they are also referred to as \textit{unicyclic graphs}, see e.g., \cite{nikiel1989pseduotree,cvetkovic1987unicyclicgraph}. Particularly, we exclude a singleton vertex being a pseudotree.	 Analogous to directed tree graphs, the following terminologies are used.
	\begin{defn}
		In a directed pseudotree $\mathcal{T}$, a vertex is called a \textbf{root}, if there is exactly one directed path from this vertex to every other vertex in $\mathcal{T}$. Furthermore, a vertex is called a \textbf{leaf} of $\mathcal{T}$, if it has no out-neighbors in $\mathcal{T}$, and a vertex is an \textbf{internal vertex} of $\mathcal{T}$, if it is neither a root nor a leaf.  We denote $\Upsilon(\mathcal{T})$ as the set that collects all the roots of a pseudotree $\mathcal{T}$. 
	\end{defn}

In Fig.~\ref{fig:pseudotrees}, typical examples of pseudotrees are presented, in which the definitions of roots, internal vertices and leaves are illustrated. Note that the class of directed pseudotrees also includes all directed rooted trees. However, different from the standard definition of trees, a pseudotree can allow for multiple roots, which form a directed circle with all the edges being oriented in the same direction, and outgoing branches from any vertex on this circle are also possible, see the right subplot in Fig.~\ref{fig:pseudotrees}. Hereafter, we will drop the word `directed' when we refer to a directed pseudotree.

	Related to the concept of vertex-disjoint paths, edge-disjoint  pseudotrees are defined as follows.
	\begin{defn}[Edge-disjoint pseudotrees] \label{def:disjointtree}
		Consider two pseudotrees $\mathcal{T}_1$ and $\mathcal{T}_2$ as subgraphs of a directed graph $\mathcal{G}$. $\mathcal{T}_1$ and $\mathcal{T}_2$ are called \textbf{disjoint} in $\mathcal{G}$ if the following two conditions hold.
		\begin{enumerate}
			\item $E(\mathcal{T}_1) \cap E(\mathcal{T}_2) = \emptyset$;
			\item $E_j \subseteq E(\mathcal{T}_1)$ or $E_j \subseteq E(\mathcal{T}_2)$, $\forall~j \in V(\mathcal{T}_1) \cup V(\mathcal{T}_2)$, where
			$
			E_j : = \{(j,i)\in E(\mathcal{T}_1) \cup E(\mathcal{T}_2) \mid i\in \mathcal{N}_j^+ \}.
			$
		\end{enumerate}
%\red{ALTERNATIVELY:
%\begin{enumerate}
%\item $\mathcal{T}_1$ and $\mathcal{T}_2$ do not share any edges;
%\item For each vertex, all outgoing edges in the set $E(\mathcal{T}_1) \cup E(\mathcal{T}_2)$ are in one and the same pseudotree.
%\end{enumerate}
%}
	\end{defn}

\begin{figure}
	%\begin{minipage}[t]{0.5\linewidth}
	\centering
	\includegraphics[scale=0.5]{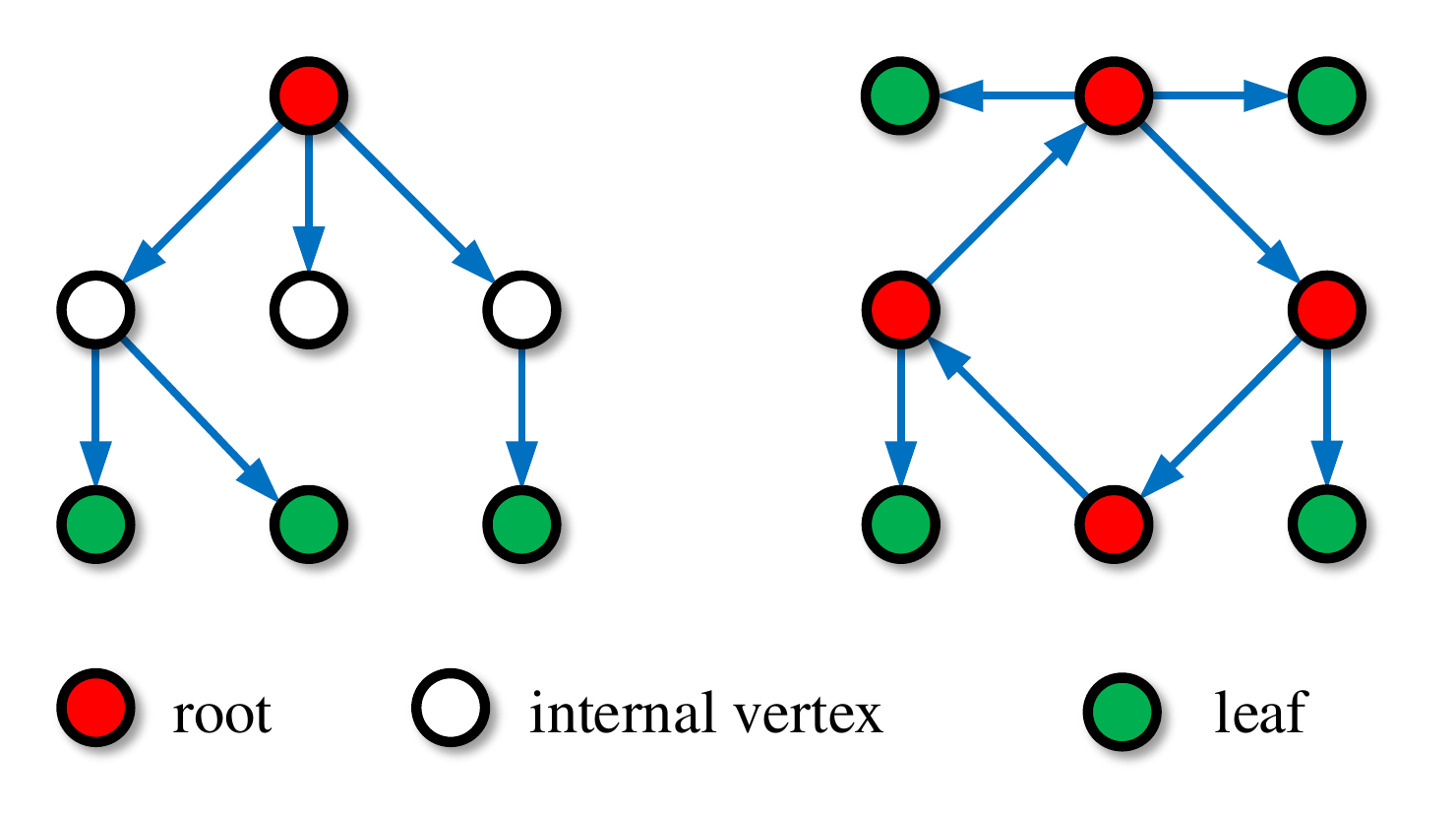}
	%\end{minipage}%
	%\begin{minipage}[t]{0.5\linewidth}
	%	\centering
	%	\includegraphics[scale=0.5]{fig/ex0pseudotree2}
	%\end{minipage}
	\caption{Typical examples of pseudotrees, in which roots, internal vertices and leaves are labeled with different colors. Note that a pseudotree may have multiple roots.}
	\label{fig:pseudotrees}
\end{figure}	
	
 The first condition means that $\mathcal{T}_1$ and $\mathcal{T}_2$ do not share any edges, while the second condition means that for each vertex, all outgoing edges in the set $V(\mathcal{T}_1) \cup V(\mathcal{T}_2)$ are in one and the same pseudotree. 
As a special case, if both $\mathcal{T}_1$ and $\mathcal{T}_2$ are directed rooted trees, then $\mathcal{T}_1$ and $\mathcal{T}_2$ do not share the same root or any common internal vertex. We illustrate the concept of disjoint pseudotrees with the following example.
	
\begin{figure}[!tp]\centering
	\centering
	\includegraphics[width=0.42\textwidth]{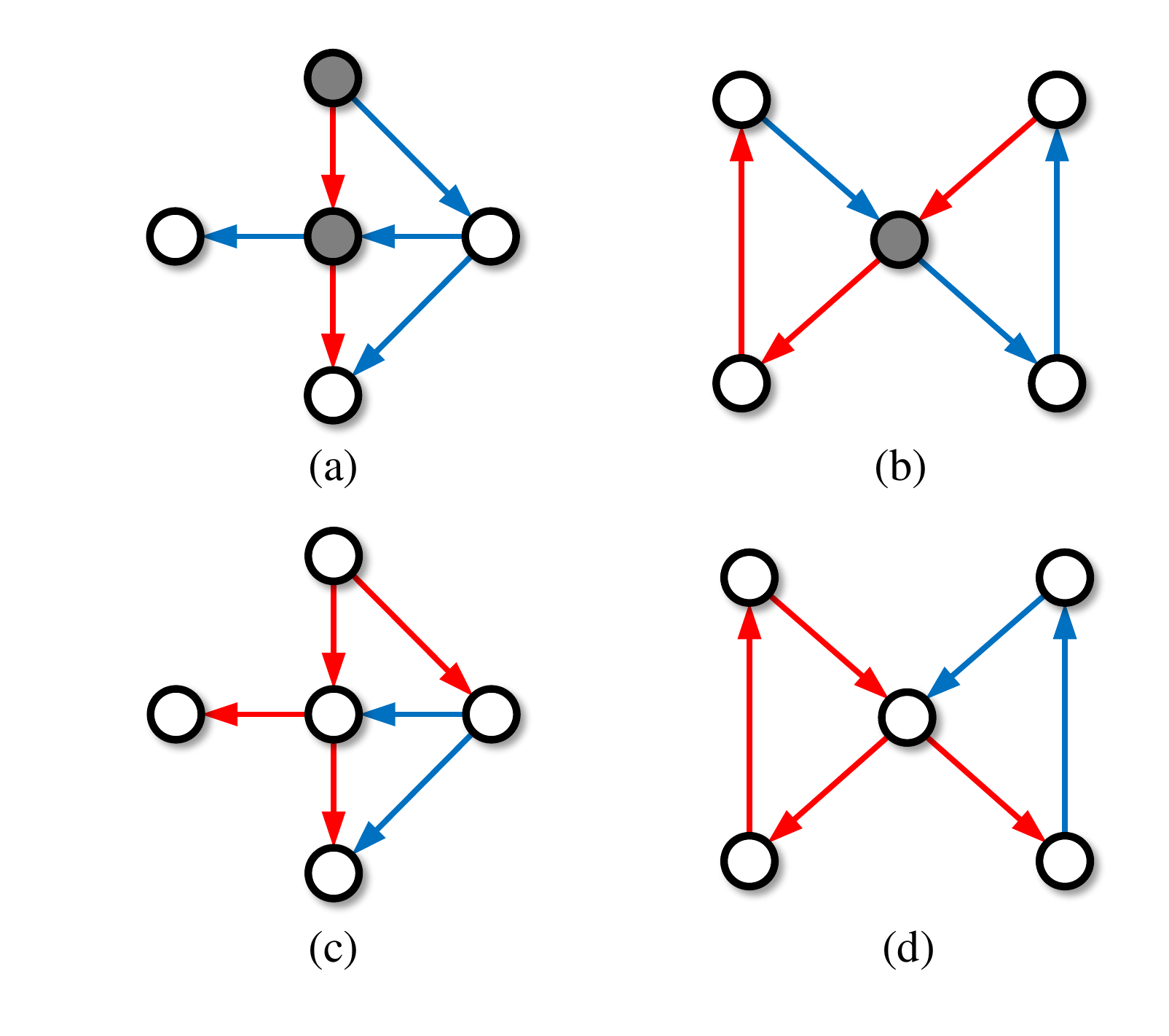}
	\caption{Illustration of disjoint pseudotrees, in which the different pseudotrees are induced by the edges with distinct colors. In (a) and (b), the pseudotrees are not disjoint, since the out-neighbors of the gray vertices are assigned
		to different pseudotrees. In contrast, the pseudotrees in (c) and (d) are characterized as disjoint pairs.}
	\label{fig:disjoint}
\end{figure}%
	
	\begin{exm}
		In Fig.~\ref{fig:disjoint}, we illustrate the conditions for disjoint pseudotrees. In (a) and (b), we decompose the directed graph into two pseudotrees, which do not share any common edges. However, they are not disjoint. In (a) and (b), the two outgoing edges of the internal vertex in the center have been assigned to different pseudotrees, which violates the second condition in Definition~\ref{def:disjointtree}. In contrast, we take a different decomposition of the two networks in (c) and (d), and then the two pseudotrees obtained in (c) and (d) become disjoint.
	\end{exm}
	
	It is worth noting that the notion of disjoint pseudotrees is closely related to that of vertex-disjoint paths. Consider $\mathcal{T}_1$ and $\mathcal{T}_2$ as two disjoint pseudotrees in $\mathcal{G}$. For any $i \in V(\mathcal{T}_1) \cap V(\mathcal{T}_2)$, if $|\mathcal{N}_i^-| \geq 2$, then there exist two in-neighbors of $i$ located in $\mathcal{T}_1$ and $\mathcal{T}_2$ separately.
Then, due to the fact that distinct pseudotrees cannot share any common root or internal vertex,
we can find two vertex-disjoint paths in the union $\mathcal{T}_1 \cup \mathcal{T}_2$ starting from two roots in $\mathcal{T}_1$ and $\mathcal{T}_2$, respectively, to two distinct in-neighbors of $i$, and each pseudotree contains exactly one path.

	Next, the concept of disjoint-edge covering for a directed graph is introduced.
	\begin{defn}[Disjoint-edge covering]
		Consider a directed graph $\mathcal{G}$, and let $\Pi: = \{\mathcal{T}_1, \mathcal{T}_2, ..., \mathcal{T}_n\}$ be a collection of connected subgraphs of $\mathcal{G}$. The edges in a set $\mathcal{E} \subseteq E(\mathcal{G})$ are {covered} by $\Pi$, if $ \mathcal{E} \subseteq E(\mathcal{T}_1) \cup E(\mathcal{T}_2) \cup ... \cup E(\mathcal{T}_n) $, and
		$\Pi$ is called a \textbf{covering} of  $\mathcal{E}$. Moreover, if all the elements in $\Pi$ are pseudotrees, which are disjoint to each other, then $\Pi$ is a \textbf{disjoint pseudotree covering} of $\mathcal{E}$.
	\end{defn}
	
	The concept of connectedness of the subgraphs is defined in Section~\ref{sec:preliminaries:graph}. Relating to the definition of disjoint pseudotree coverings, the following two lemmas are given.
	\begin{lem} \label{lem:treesmost}
		For a directed simple graph $\mathcal{G}$ with $|V(\mathcal{G})| \geq 2$, there always exists a set of disjoint pseudotrees that cover all the edges in $E(\mathcal{G})$ or any subset of $E(\mathcal{G})$.
	\end{lem}
	\begin{proof}
		To prove this statement, we consider each vertex $j \in V(\mathcal{G}) \setminus {\mathcal{S}_\mathrm{in}}(\mathcal{G})$, with ${\mathcal{S}_\mathrm{in}}(\mathcal{G})$ the set of all the sinks of $\mathcal{G}$. Starting from $j$, we can construct a directed star tree (a special type of pseudotrees) with $j$ as the single root and the vertices in $\mathcal{N}_j^+$ as the leaves. Then, $|V(\mathcal{G}) \setminus {\mathcal{S}_\mathrm{in}}(\mathcal{G})|$ pseudotrees are formed as a covering of $E(\mathcal{G})$, which are disjoint, since any two trees do not share a common root or any common internal vertex. For any subset of $E(\mathcal{G})$, its disjoint pseudotree covering can be found using the similar approach.
	\end{proof}
	
	Let us define a \textit{minimal} pseudotree, which only contains one root and all the out-neighbors of this root. By the proof of Lemma~\ref{lem:treesmost}, the maximal number of disjoint pseudotrees that coexist in $\mathcal{G}$ is $|V(\mathcal{G})\setminus {\mathcal{S}_\mathrm{in}}(\mathcal{G})|$. 
	Then, the following lemma holds.
	
	\begin{lem} \label{lem:treemore}
	Let $\mathcal{G}$ be a simple directed graph. If there exist $k_1$ disjoint pseudotrees covering $E(\mathcal{G})$, with $k_1 < |V(\mathcal{G})\setminus {\mathcal{S}_\mathrm{in}}(\mathcal{G})|$, then there also exist $k_2$ disjoint pseudotrees, for any $k_1 < k_2 \leq |V(\mathcal{G}) \setminus {\mathcal{S}_\mathrm{in}}(\mathcal{G})|$, that cover $E(\mathcal{G})$.
	\end{lem}
	\begin{proof}
		The maximal number of disjoint pseudotrees that coexist in $\mathcal{G}$ does not exceed $|V(\mathcal{G})\setminus {\mathcal{S}_\mathrm{in}}(\mathcal{G})|$, where ${\mathcal{S}_\mathrm{in}}(\mathcal{G})$ is the set of the sinks in $\mathcal{G}$. It then requires $k_1 < |V(\mathcal{G}) \setminus {\mathcal{S}_\mathrm{in}}(\mathcal{G})|$, implying that in the $k_1$ disjoint pseudotrees,
		there exists at least one pseudotree $\mathcal{T}_k$ which contains at least one internal vertex or contains multiple roots. In both cases, we will show that $\mathcal{T}_k$ can be decomposed into two disjoint pseudotrees.
		
		Suppose $\mathcal{T}_k$ contains internal vertices. We can always find an internal vertex $i$ with all its out-neighbors being the leaves of $\mathcal{T}_k$.
		Define a directed tree ${T}_a$ with $i$ as the root and $\mathcal{N}_i^+$ as the leaves. Thereby, $\mathcal{T}_k$ is decomposed into two a directed tree ${T}_a$ and a pseudotree $\mathcal{T}_b$, where $R({T}_b) : = \Upsilon(\mathcal{T}_k)$, $V(\mathcal{T}_b) \subseteq (\mathcal{T}_k)$, and $E({T}_b): = E(\mathcal{T}_k) \setminus E({T}_a)$. Note that ${T}_a$ and $\mathcal{T}_b$ are disjoint by Definition~\ref{def:disjointtree}. Moreover, since ${T}_a$ and $\mathcal{T}_b$ are subgraphs of $\mathcal{T}_k$, which is disjoint to the other trees, ${T}_a$ and $\mathcal{T}_b$ are also disjoint to the other pseudotrees. Next, suppose $\mathcal{T}_k$ does not contain any internal vertex but multiple roots, i.e., $|\Upsilon(\mathcal{T}_k)| \geq 2$. In this case, we define the directed tree ${T}_a$, which is rooted at one of $\Upsilon(\mathcal{T}_k)$ and includes all the out-neighbors of this root as the leaves of ${T}_a$. Then, similar to the previous case, we can partition $\mathcal{T}_k$ into two disjoint pseudotrees, which are disjoint to the other pseudotrees in $\mathcal{G}$.
		Therefore, in the above cases, $\mathcal{E}$ can be covered by $k_1 + 1$ disjoint pseudotrees. The statement of this lemma follows by iteratively applying the above reasoning for all $k_2 \geq k_1 + 1$.
	\end{proof}

	\section{Allocation of Excitation Signals}
	\label{sec:result}
	
	On the basis of disjoint pseudotree covering, we present a novel approach for the allocation of excitation signals such that the generic identifiability of a network model set $\mathcal{M}$ is achieved. The key step relies on a partitioning of the extended graph $\widehat{\mathcal G}$ into a minimal number of disjoint pseudotrees.
	
	\subsection{Generic Identifiability: A  Pseudotree Characterization}
	
	From Section~\ref{sec:pseudotree}, we notice that there is a clear association between vertex-disjoint paths and disjoint pseudotrees.
	Thus, this section provides a novel characterization for  generic identifiability using the concept of disjoint pseudotrees, which is used as the theoretical foundation for the follow-up synthesis method.

	\begin{thm} \label{thm:treecover}
		Consider a  network model set $\mathcal{M}$ defined in \eqref{eq:modset}, which satisfies the conditions of Lemma \ref{lem1} and Assumptions \ref{ass2} and \ref{ass3}.
		Let $\widehat{\mathcal{G}}$ be its extended graph with parameterized edges set $E(\widehat{\mathcal{G}})$ and the set of stimulated vertices $\mathcal{U} = \{\tau_1, \tau_2, ..., \tau_{|\mathcal{U}|}\} \subseteq V(\widehat{\mathcal{G}})$, where $|\mathcal{U}| = K + p$.
		Then, the network model set $\mathcal{M}$ is generically identifiable from $(r, w)$ \textbf{if and only if}  
        there exists a disjoint pseudotree covering of $E(\widehat{\mathcal{G}})$, denoted by $\Pi = \{\mathcal{T}_1, \mathcal{T}_2, ..., \mathcal{T}_{n}\}$ with $n \geq |\mathcal{U}|$, such that $\tau_k \in \Upsilon(\mathcal{T}_k)$, $\forall~k \in \{1, 2, ..., |\mathcal{U}|\}$, and $b_{\mathcal{U} \rightarrow \widehat{\mathcal{P}}_j} = |\widehat{\mathcal{P}}_j|$, $\forall~j \in V(\mathcal{T}_{|\mathcal{U}|+1})\cup ... \cup V(\mathcal{T}_{n})$.       
		Here, $\Upsilon(\mathcal{T}_k)$ is the set of roots in the pseudotree $\mathcal{T}_k$, and ${b}_{\mathcal{U} \rightarrow \widehat{\mathcal{P}}_j}$ denotes the  maximum number of mutually vertex-disjoint
		paths from $\mathcal{U}$ to $\widehat{\mathcal{P}}_j$.
	\end{thm}

	\begin{proof}
		We first prove the `if' statement.  Let $\Pi = \{\mathcal{T}_1, \mathcal{T}_2, ..., \mathcal{T}_n\}$, with $n > |\mathcal{U}|$, be a set of pseudotrees that cover all the parameterized edges in $\widehat{\mathcal G}$. Note that the disjointness of the pseudotrees in Definition~\ref{def:disjointtree} implies that 
		the paths in different disjoint pseudotrees are vertex-disjoint, if they have no common starting
		or ending nodes, and, for any vertex $j \in V(\widehat{\mathcal G})$, all the edges incident from the vertices in $\widehat{\mathcal{P}}_j$ to $j$ should belong to distinct pseudotrees. Furthermore, any two disjoint pseudotrees cannot share common root nodes, and thus $\tau_i \ne \tau_j$, for all $i\ne j$.
		Consequently, the above properties of disjoint pseduotrees yield that
		there exist   $|\widehat{\mathcal{P}}_j|$ vertex-disjoint paths from $\{\tau_1, \tau_2, ..., \tau_{n}\}$ to $\widehat{\mathcal{P}}_j$. Define 
		$\bar{\mathcal{V}}: = V(\mathcal{T}_{|\mathcal{U}|+1})\cup \cdots \cup V(\mathcal{T}_{n})$ such that
		 all the in-coming edges of each vertex $j \in V(\widehat{\mathcal G}) \setminus \bar{\mathcal{V}}$ belong to distinct pseudotrees, and there always exist at least $|\widehat{\mathcal{P}}_j|$ vertex-disjoint paths from $\mathcal{U}$ to $\widehat{\mathcal{P}}_j$.
		Since each $\tau_i$, which is a root of the pseudotree $\mathcal{T}_k$, $k = 1,2,...,|\mathcal{U}|$, is chosen as stimulated vertex affected by an independent stimulation source, namely, either a white noise or a designed external excitation signal, then the equation \eqref{eq:maxpath} holds for all vertex $j \in V(\widehat{\mathcal G}) \setminus \bar{\mathcal{V}}$.
		For the rest of vertices in the set $\bar{\mathcal{V}}$, \eqref{eq:maxpath} is also satisfied due to $b_{\mathcal{U} \rightarrow \widehat{\mathcal{P}}_j} = |\widehat{\mathcal{P}}_j|$, $\forall~j \in \bar{V}$.
		It then follows from Lemma~\ref{lem:genericid} that the network model set $\mathcal{M}$ is generically identifiable. 
		
%	 
%		{\red Next, the `only if' statement is proven. Let the network model set $\mathcal{M}$ be generically identifiable, and we proceed the proof by contradiction.
%		Assume that there does not exist a set of disjoint pseudotrees covering all parameterized edges in $E(\widehat{\mathcal{G}})$. Note that we can always find a disjoint pseudotree covering of $E(\widehat{\mathcal{G}})$, following Lemma~\ref{lem:treesmost}. Thus, the assumption holds if there exists at least one vertex $j \in V(\mathcal{T}_{|\mathcal{U}|+1})\cup \cdots \cup V(\mathcal{T}_{n})$, which does not satisfy \eqref{eq:maxpath}. Consequently, $\mathcal{M}$ is not generically identifiable, which causes a contradiction.}
	
 Next, the `only if' statement is proven. Let the network model set $\mathcal{M}$ be generically identifiable, and we will show that a disjoint pseudotree covering exists and satisfies the condition in this theorem. It is obtained from Lemma~\ref{lem:treesmost} that we can always find a disjoint pseudotree covering of $E(\widehat{\mathcal{G}})$, denoted by $\Pi = \{\mathcal{T}_1, \mathcal{T}_2, ..., \mathcal{T}_{n}\}$, with $n = |V(\mathcal{G}) \setminus {\mathcal{S}_\mathrm{in}}(\mathcal{G})|$, where each pseudotree is only composed of a node as its root and all its out-neighbors as leaves. As the nodes in $\{1, 2, ..., |\mathcal{U}|\}$ are excited by external signals, we have a set of pseudotrees $\Pi_r \subseteq \Pi$, with $|\Pi_r| = |\mathcal{U}|$, in which every pseudotree has its root excited.  Then, we only need to prove that  the path condition $b_{\mathcal{U} \rightarrow \widehat{\mathcal{P}}_j} = |\widehat{\mathcal{P}}_j|$ holds, for every node $j \in \Pi \setminus \Pi_r$.
	This is guaranteed by generic identifiability of $\mathcal{M}$ from Lemma~\ref{lem:genericid}.
		
		That completes the proof.
	\end{proof}
 Following Theorem~\ref{thm:treecover}, a sufficient condition for generic identifiability can be obtained.
	\begin{coro} \label{coro:treecover}
		Consider a  network model set $\mathcal{M}$ defined in \eqref{eq:modset}, which satisfies the conditions of Lemma \ref{lem1} and Assumptions~\ref{ass2} and \ref{ass3}.
		Let $\widehat{\mathcal{G}}$ be its extended graph, with the set of parameterized edges $E(\widehat{\mathcal{G}})$.
		Then, the network model set $\mathcal{M}$ is generically identifiable from $(r, w)$ \textbf{if}  
        there exists a set of disjoint pseudotrees covering all the elements in $E(\widehat{\mathcal{G}})$, and each pseudotree has at least one root vertex being excited.
	\end{coro}
	\begin{proof}
	    The condition in this corollary implies that the cardinality of the covering $n$ is less than or equal to $|\mathcal{U}|$ defined in Theorem~\ref{thm:treecover}. It then follows from Lemma~\ref{lem:treemore} that if $E(\widehat{\mathcal{G}})$ can be covered a set of $n$ disjoint pseudotrees, we can construct $\tilde{n}$ disjoint pseudotrees, where $\tilde{n} > |\mathcal{U}|$, to cover $E(\widehat{\mathcal{G}})$.  As a result, the proof can be proceeded following a similar reasoning as Theorem~\ref{thm:treecover} and therefore is omitted here.
	\end{proof}
	
	The condition in Corollary~\ref{coro:treecover} requires that in a given disjoint pseudotree covering of $E(\widehat{\mathcal{G}})$, one of the roots of each pseudotree is a stimulated vertex. This condition is  sufficient for generic identifiability. The condition in Theorem~\ref{thm:treecover} is needed when we have more disjoint pseudotrees in a covering than the number of stimulated vertices in $\widehat{\mathcal{G}}$. In this case, only a partial number of pseudotrees contains stimulated vertices in their roots, while the vertices in the remaining set of pseudotrees need to satisfy the path condition in \eqref{eq:maxpath}, which requires based on the full topology of $\widehat{\mathcal{G}}$.

	Compared to Lemma~\ref{lem:genericid}, Theorem~\ref{thm:treecover} and Corollary~\ref{coro:treecover} provide more integrated conditions for characterizing the generic identifiability. The major advantage of this pseudotree covering condition in Corollary~\ref{coro:treecover} over the path-based conditions in e.g. \cite{hendrickx2018identifiability,weerts2018single} is that, rather than providing a vertex-wise analysis, it has the potential for the synthesis problem we are interested in. Particularly, combining with Theorem~\ref{thm:treecover}, we obtain a useful tool for allocating the minimal number of excitation signals to achieve the generic identifiability of the overall network.

	\begin{coro} \label{lem:bound}
		The minimal number $K$ of external excitation signals that guarantees the generic identifiability of a directed network model set $\mathcal{M}$ is bounded as
		\begin{equation} \label{eq:bound}
		\max\left\{|\mathcal{S}_\mathrm{ou}(\widehat{\mathcal{G}})|, \max_{j \in V(\widehat{\mathcal{G}})} {|\widehat{\mathcal{P}}_j|}\right\}
		- p
		\leq K
		\leq \kappa(\widehat{\mathcal{G}}) - p,
		\end{equation}
		where $\kappa(\widehat{\mathcal{G}})$ is the minimal number of disjoint pseudotrees that cover all the edges of $\widehat{\mathcal{G}}$.
	\end{coro}
	\begin{proof}
		The lower bound is obtained immediately from Lemma~\ref{lem:genericid} as a necessary number of external excitation signals that are required for the sources and the other vertices. The upper bound then results from applying Theorem~\ref{thm:treecover}, and it suffices to assign an independent external signal to a root of each pseudotree to achieve generic identifiability.
	\end{proof}
	
	%\begin{rem}
	%    Add some results for special graphs, e.g., complete graphs, bipartite graphs, or...
	%\end{rem}
	
	The upper bound in \eqref{eq:bound} plays a central role in this paper since it directly implies that solving the synthesis problem amounts to finding the minimal number of disjoint pseudotrees in the network that cover all the parameterized edges in $E(\widehat{\mathcal{G}})$. At this point, we relate the synthesis problem to a combinatorial optimization problem.
	
	\begin{exm}
		Consider the five-vertex network in Fig.~\ref{fig:5nodenet}, and we find that the parameterized edges of the extended graph in Fig.~\ref{fig:ex5nodenet} can be covered by five disjoint pseudotrees as shown in Fig.~\ref{fig:ex5decomp}. Observe that there is a unique stimulated vertex in each pseudotree, which is a root. Thus, the condition  in Theorem~\ref{thm:treecover} is satisfied, and we  conclude that the dynamic network model set $\mathcal{M}$ in Example~\ref{exam1} is generically identifiable.
		\begin{figure}[!tp]\centering
			%        \begin{minipage}[t]{\linewidth}
			\includegraphics[scale=.33]{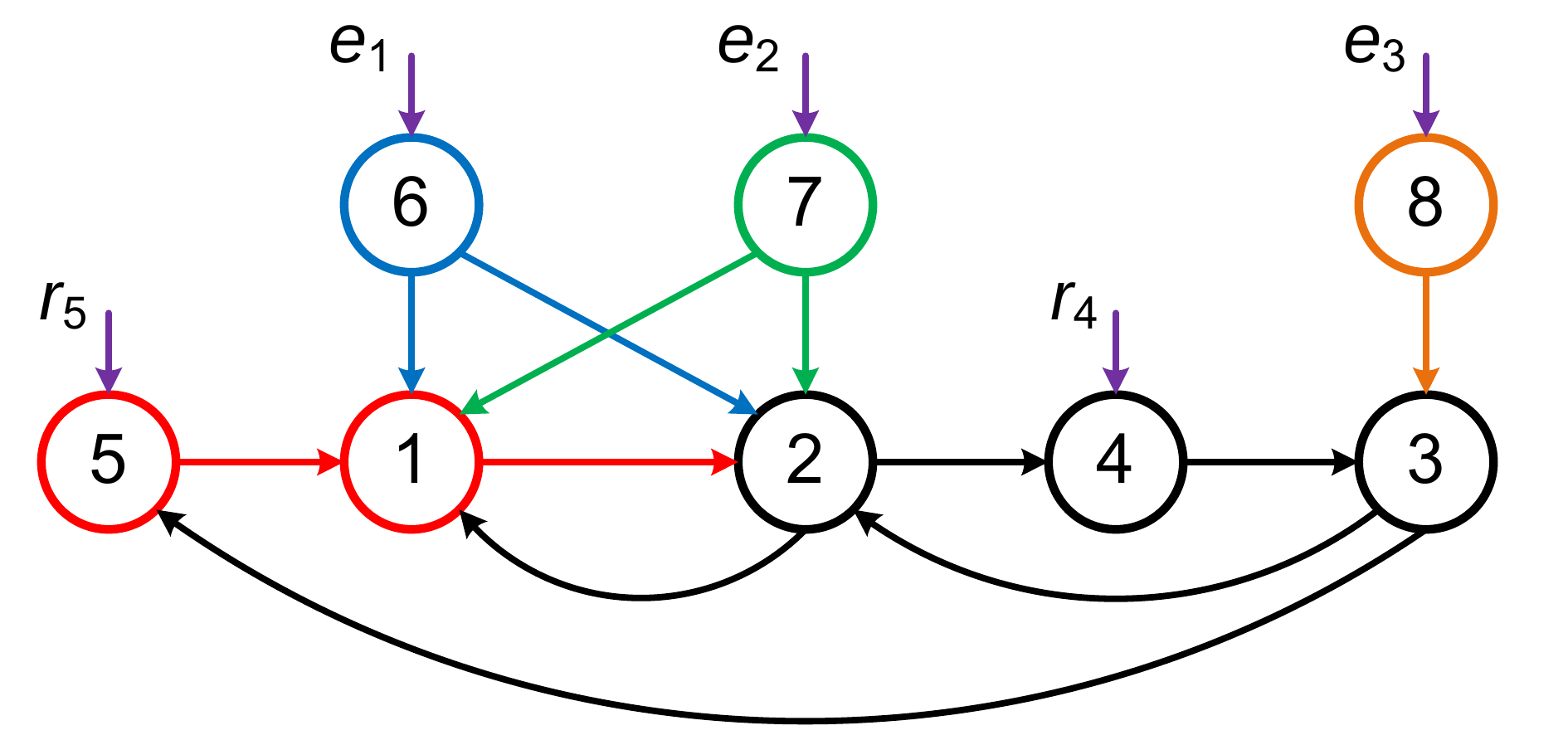}
			\centering
			%        \subcaption{}
			%        \end{minipage}
			%        \begin{minipage}[t]{\linewidth}
			%        \includegraphics[scale=.38]{fig/ex3extendeddecomp1}
			%        \centering
			%        \label{fig:ex5decomp1}
			%        \subcaption{}
			%        \end{minipage}
			\caption{The extended graph in 	Fig.~\ref{fig:ex5nodenet} is decomposed into 5 disjoint pseudotrees, which are highlighted with different colors. Since all the parameterized edges are covered, and each stimulated vertex is located at a root of each pseudotree, the network in Fig.~\ref{fig:5nodenet} is generically identifiable.}
			\label{fig:ex5decomp}
		\end{figure}
		\begin{figure}[!tp]\centering	
			\includegraphics[scale=.33]{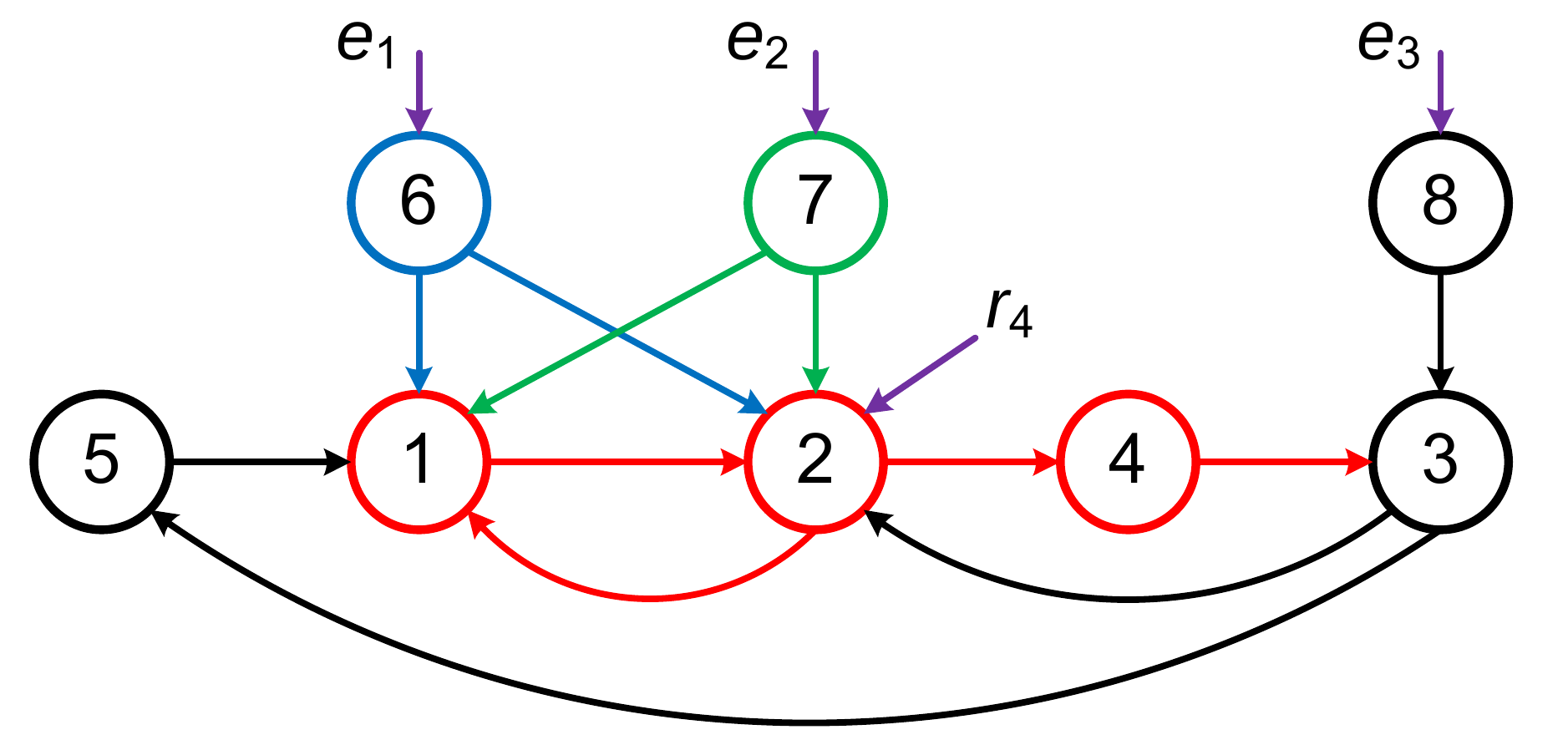}
			\centering
			\caption{Four disjoint pseudotrees are needed to cover all the parameterized edges of the extended graph in Fig.~\ref{fig:ex5nodenet}. Thus, in addition to the white noise excitatons $e_1$, $e_2$ and $e_3$, only one external excitation signal is required to achieve generic identifiability of the network in Fig.~\ref{fig:5nodenet}, and assigning this excitation signal to either vertex 1 or 2 will lead to this result.}
			\label{fig:ex5decomp1}
		\end{figure}
	\end{exm}

	For a simple network consisting of only a few vertices, e.g., Fig.~\ref{fig:5nodenet}, we may immediately obtain the minimal number of excitation signals and their locations such that generic identifiability is achieved, see Fig.~\ref{fig:ex5decomp1}.  However, when a more complicated graph is considered, a systematic approach is required to decompose a graph into a minimal number of disjoint pseudotrees. Thus, in the next section, we focus on an algorithmic procedure to tackle this combinatorial problem.

	\subsection{Excitation Allocation: A Pseudotree Merging Approach}
	
	In this section, we aim to solve an excitation allocation problem, which aims for a minimal number of external excitation signals which are used to guarantee generic identifiability of a network model set.
%consistently estimate the proper transfer functions associated with all the edges of interest in a dynamic network. 	
To this end, a two-step scheme is developed, where the steps correspond to the conditions in Corollary~\ref{coro:treecover} and Theorem~\ref{thm:treecover}, respectively. In the first step, we devise a heuristic method to find a minimum number of disjoint pseudotrees covering all the edges of the extended graph $\widehat{\mathcal G}$. Then, the second step is to allocate excitation signals at the roots of some selected pseudotrees in the covering such that generic identifiability is achieved. Hereafter, we present the detailed implementation for the two steps.

%	Theorem~\ref{thm:treecover} shows the relation between generic identifiability and a disjoint pseudotree covering, and \eqref{eq:bound} reveals the upper bound of the minimal number of excitation signals that are needed for generic identifiability. Thus, for a given dynamic network model set, the \red{smallest number of}  disjoint pseudotrees that can be found to cover all of its parameterized edges potentially implies that the fewer excitation signals that are required to identify all the parameterized modules. Based on this argument, we tackle the following graph theoretical problem as the first step: \textit{Given a directed graph $\widehat{\mathcal{G}}$, find a set of disjoint pseudotree covering $\Pi = \{\mathcal{T}_1, \mathcal{T}_2, ..., \mathcal{T}_n\}$ such that all the parameterized edges of $\widehat{\mathcal{G}}$ are covered by $\Pi$, and $|\Pi|$ is minimized}.
	
\subsubsection{Pseudotree Covering}	
According to \eqref{eq:bound}, the smallest number of disjoint pseudotrees that can be found to cover all of the edges
potentially induces the smallest number of  excitation signals that is required to identify all the modules. Based on this argument, we tackle the following graph-theoretical problem as the first step: \textit{Given a directed graph $\widehat{\mathcal{G}}$, find a set of disjoint pseudotree covering $\Pi = \{\mathcal{T}_1, \mathcal{T}_2, ..., \mathcal{T}_n\}$ such that all the edges of $\widehat{\mathcal{G}}$ are covered by $\Pi$, and $|\Pi|$ is minimized}.

	To efficiently solve this minimal covering problem, we devise a graph merging algorithm.  Lemma \ref{lem:treesmost} indicates that for any directed graph $\widehat{\mathcal{G}}$, we can always find a disjoint minimal pseudotree covering,
	\begin{equation} \label{eq:Pi0}
	\Pi_0 = \{\mathcal{T}_1^{(0)}, \mathcal{T}_2^{(0)}, ..., \mathcal{T}_{|\Pi_0|}^{(0)}\},
	\end{equation}
	where each minimal pseudotree is rooted at a vertex in $V(\widehat{\mathcal{G}}) \setminus {\mathcal{S}_\mathrm{in}}(\widehat{\mathcal{G}})$, with ${\mathcal{S}_\mathrm{in}}$ the set of the sinks of $\widehat{\mathcal{G}}$. Here, $|\Pi_0| =   |V(\widehat{\mathcal{G}})| - |{\mathcal{S}_\mathrm{in}}(\widehat{\mathcal{G}})|$. In other words, each vertex, besides the sinks, is the root of its own pseudotree, consisting of all links that connect the vertex itself to all of its out-neighbors. The proposed approach starts with $\Pi_0$ as the initial disjoint pseudotree covering, and we then implement a specific strategy to recursively merge the pseudotrees until there are no mergeable pseudotrees in a covering.

	As a relevant and necessary concept, the \textit{mergeability} of pseudotrees is defined as follows.
	\begin{defn}[Mergeability]
		\label{def:Mergeability}
		Consider two disjoint pseudotrees $\mathcal{T}_1$ and $\mathcal{T}_2$ and $V(\mathcal T_1) \cap V(\mathcal T_2) \ne \emptyset$. We say $\mathcal{T}_1$ is mergeable to $\mathcal{T}_2$, if
		\begin{enumerate}
			\item the union of $\mathcal{T}_1$ and $\mathcal{T}_2$, i.e., $(V(\mathcal{T}_1) \cup V(\mathcal{T}_2), E(\mathcal{T}_1) \cup E(\mathcal{T}_2))$ is also a pseudotree,
			
			\item and
			there is a directed path from every vertex $i \in \Upsilon(\mathcal{T}_2)$ to every vertex $j \in V(\mathcal{T}_1)$.
			\\
		\end{enumerate}
	\end{defn}
If $\mathcal{T}_1$ is mergeable to $\mathcal{T}_2$ then the roots of $\mathcal{T}_2$ remain the roots of the merged pseudotree.
	The mergeability of a pseudotree $\mathcal{T}_1$ to $\mathcal{T}_2$ requires that $\mathcal{T}_1$ and $\mathcal{T}_2$ do not share any common leaf and internal vertex. As a result, merging $\mathcal{T}_1$ and $\mathcal{T}_2$ yields a new pseudotree $\mathcal{T}_3$, where $\Upsilon(\mathcal{T}_3) \supseteq \Upsilon(\mathcal{T}_2)$.
	Note that $\mathcal{T}_1$ being mergeable to $\mathcal{T}_2$ does not necessarily mean that $\mathcal{T}_2$ is also mergeable to $\mathcal{T}_1$.
	Next we introduce an algebraic characterization of a given disjoint pseudotree covering, that will be  instrumental in our follow-up merging approach.
	\begin{defn} \label{def:MergeMat}
		Denote a set
		\begin{equation}
			\mathbb{M} = \{1, 0, \varnothing\}.
		\end{equation}
		Let $\Pi = \{\mathcal{T}_1, \mathcal{T}_2, ..., \mathcal{T}_n\}$ be a disjoint pseudotree covering of a directed graph. The \textbf{characteristic matrix} of $\Pi$ is denoted by $\mathscr{M} \in \mathbb{M}^{n \times n}$, whose $(i,j)$-th entry is defined as
		\begin{equation} \label{eq:Mergeability}
		\mathscr{M}_{ij} = \begin{cases}
		1 & \text{if $\mathcal T_i$ is mergeable to $ \mathcal T_j$;} \\
		\varnothing &  \text{if $V(\mathcal T_j) \cap V(\mathcal T_i) = \emptyset$;}\\
		0 &  \text{otherwise.}\\
		\end{cases}
		\end{equation}
	\end{defn}

The characteristic matrix of the initial pseudotree covering $\Pi_0$ \eqref{eq:Pi0} is denoted by $\mathscr{M}^{(0)}$.
The relation between $\mathscr{M}^{(0)}$ and the adjacency matrix of $\widehat{\mathcal{G}}$ is now discussed. 	Let $A(\widehat{\mathcal{G}}) \in \mathbb{R}^ {(L+p)\times (L+p)}$ be the adjacency matrix of the directed graph $\widehat{\mathcal{G}}$ such that $[A(\widehat{\mathcal{G}})]_{{i}j} = 1$ if $({j},i) \in E (\widehat{\mathcal{G}})$, and $[A(\widehat{\mathcal{G}})]_{{i}j} = 0$ otherwise. Without loss of generality, we assume that $A(\widehat{\mathcal{G}})$ is permuted such that all zero columns corresponding to $\mathcal{S}_\mathrm{in}(\widehat{\mathcal{G}})$ are its last columns. Then, the following result holds.
\begin{lem}
Given a graph $\widehat{\mathcal{G}}$ with the adjacency matrix $A(\widehat{\mathcal{G}})$. Denote
\begin{equation*}
a_{ij} ={\left([A(\widehat{\mathcal{G}})+I\mathfrak{i}]_{\star i}\right)^\top} [A(\widehat{\mathcal{G}})+I \mathfrak{i}]_{\star j}, \label{eq:1}
\end{equation*}
where ${i}$, $j \in {1, 2, ..., |\Pi_0|}$, $\mathfrak{i}$ denotes the imaginary unit, and $[A(\widehat{\mathcal{G}})+I \mathfrak{i}]_{\star {i}}$ indicates the ${i}$-th column of $A(\widehat{\mathcal{G}})+I \mathfrak{i}$. The characteristic matrix $\mathscr{M}^{(0)}$ of $\Pi_0$ in \eqref{eq:Pi0} is formulated as follows: $\mathscr{M}^{(0)}_{ii} = 0$ for all ${i}$, while for $j\not={i}$:
\begin{equation*}
\mathscr{M}^{(0)}_{ij} = \begin{cases}
1, \quad \operatorname{Re}(a_{{i}j})=0,\text{ {and} }  \operatorname{Im}(a_{{i}j}) \not= 0, \text{ {and} }\\ \quad \quad [A(\widehat{\mathcal{G}})]_{{i}j} \not=0. \\
0, \quad \operatorname{Re}(a_{{i}j}) \not= 0 \text{ or } \{ \operatorname{Re}(a_{{i}j})=0, \text{ {and} }\\
\quad \quad \operatorname{Im}(a_{{i}j}) \not= 0, \text{ {and} } [A(\widehat{\mathcal{G}})]_{{i}j} =0 \}. \\
\varnothing, \quad a_{{i}j} = 0,
\end{cases} \label{eq:2}
\end{equation*}
where $\operatorname{Re}(\cdot)$ and $\operatorname{Im}(\cdot)$ denote the real and imaginary parts of a complex number.
\end{lem}
\begin{proof}
The matrix $\mathscr{M}^{(0)}$ is of the size $|\Pi_0| \times |\Pi_0|$, and its $k$-th row or column corresponds 
to the pseudotree $\mathcal{T}_k$, which consists of the $k$-th vertex in $\widehat{\mathcal{G}}$ as the root and all the out-neighbors of the $k$-th vertex.
Since vertex $i$ cannot be merged to vertex $i$, it is obvious that $\mathscr{M}^{(0)}_{ii}=0$.
	
The condition $V(\mathcal T_j) \cap V(\mathcal T_i) = \emptyset$ in (\ref{eq:Mergeability}) is equivalent to the situation that (i) there is no directed edge between $j$ and $i$ (in either direction), and (ii) nodes $i$ and $j$ do not share any out-neighbors in $\widehat{\mathcal{G}}$. Note that condition (i) is equivalent to $[A(\widehat{\mathcal{G}})]_{ij} = [A(\widehat{\mathcal{G}})]_{ji}=0$, and that according to condition (ii) there does not exist a node $k$ such that $[A(\widehat{\mathcal{G}})]_{kj}\neq 0$ and $[A(\widehat{\mathcal{G}})]_{ki}\neq 0$, for all $k \neq i,j$. For $i \neq j$ is follows then that conditions (i) and (ii) are equivalent to $a_{ij}=0$, showing that in this situation $\mathscr{M}^{(0)}_{ij}=\varnothing$.

For the minimal pseudotree covering $\Pi_0$, $\mathcal{T}_i$ is mergeable to $\mathcal{T}_j$ if nodes $i$ and $j$ do not share a common out-neighbor, and if there exists a directed edge from node $j$ to node $i$. The case $\{\operatorname{Re}(a_{ij})=0$ and $\operatorname{Im}(a_{ij}) \not= 0\}$ represents the situation that nodes $i$ and $j$ do not have a common out-neighbor, while there exists a directed edge between $i$ and $j$ (in either direction). For mergeability of $\mathcal{T}_i$ into $\mathcal{T}_j$ a directed edge needs to be present from node $j$ to node $i$, which is guaranteed by the additional requirement that $[A(\widehat{\mathcal{G}})]_{ij}\neq 0$. This proves the
situation $\mathscr{M}^{(0)}_{ij}=1$.
The situation $\mathscr{M}^{(0)}_{ij} = 0$ appears in the remaining cases.
\end{proof}

	Having the characteristic matrix of $\Pi_0$, the following notations and operators are defined to merge the initial  pseudotrees. Define $\mathscr{M} \in \mathbb{M}^{|\Pi| \times |\Pi|}$, and let $\mathscr{M}_{i\star}$ and $ \mathscr{M}_{\star j}$ be the $i$-th row and $j$-th column of a matrix $\mathscr{M} \in \mathbb{M}$, where $\Pi$ is a disjoint pseudotree covering. To feature the merging of two pseudotrees from  an algebraic point of view, we define a commutative operator
	\begin{equation}\label{eq:operation}
	c = a \odot b = b \odot a,
	\end{equation}
	with $a ,b, c \in \mathbb{M}$, which follows the rules:
	\begin{align} \label{eq:rule}
	&    1 \odot 1 = 1, \ 1 \odot 0 = 0, \ 1 \odot \varnothing = 1,  \nonumber\\
	&    0 \odot 0 = 0, \ \varnothing \odot 0 = 0, \ \varnothing \odot \varnothing = \varnothing.
	\end{align}
	Furthermore, we also extend this above operators to vectors in $\mathbb{M}^n$. Let $\rho, \mu \in \mathbb{M}^n$ be two column (or row) vectors. Then,  $\rho \odot \mu = \mu \odot \rho$ stands for an entrywise operator that returns a new column (or row) vector, whose $i$-th element is given by $\rho_i \odot \mu_i$. For a given disjoint pseudotree covering $\Pi$ with $|\Pi| = n$ and a set $\mathbb{N}: = \{1 ,2, ..., n\}$, we then define the following function
	\begin{equation}  \label{eq:func}
	\mathscr{F}: \mathbb{M}^{n\times n} \times \mathbb{N} \times \mathbb{N} \rightarrow \mathbb{M}^{(n-1)\times (n-1)},
	\end{equation}
	and $\hat{\mathscr{M}} = \mathscr{F} (\mathscr{M}, i, j)$ is a reduction of $\mathscr{M}$ obtained by the following algebraic operations:
	\begin{enumerate}
		\item $\hat{\mathscr{M}} = \mathscr{M}$;
		\item Row merging: $ \hat{\mathscr{M}}_{j\star} =  \mathscr{M}_{i\star} \odot \mathscr{M}_{j\star}$;
		\item Column merging: $ \hat{\mathscr{M}}_{\star j} =  \mathscr{M}_{\star i} \odot  \mathscr{M}_{\star j}$;
		\item Remove $i$-th row and column of $\hat{\mathscr{M}}$.
	\end{enumerate}
As will be shown next, this operation conforms to the merging of the $i$-th pseudotree into the $j$-th one.	
	Note that the order of the row and column operations can be switched, which will not affect the outcome $\hat{\mathscr{M}}$.
	\begin{thm}
		Consider a directed graph $\widehat{\mathcal{G}}$, and let $\Pi$ be a disjoint pseudotree covering of all the edges of $\widehat{\mathcal{G}}$ where the characteristic matrix is $\mathscr{M}$. Suppose in $\Pi$, the $i$-th pseudotree is mergeable to the $j$-th one. Let $\hat{\Pi}$, with $|\hat{\Pi}| = |\Pi| - 1$, be a new covering obtained by merging the $i$-th pseudotree into the $j$-th one. Then the characteristic matrix of $\hat{\Pi}$ is given as $\hat{\mathscr{M}} = \mathscr{F} (\mathscr{M}, i, j)$.
	\end{thm}
	\begin{proof} 
		We first show that the rules in \eqref{eq:rule} are consistent with merging two disjoint pseudotrees in a covering. Let a pseudotree $\mathcal{T}_1$ be mergeable to $\mathcal{T}_2$. Then, the following statements hold due to Definition~\ref{def:Mergeability}:
		\begin{enumerate}
			\item If either $\mathcal{T}_1$ or $\mathcal{T}_2$ cannot merge (be merged to) any other pseudotree $\mathcal{T}_3$ in $\Pi$, then  the union of $\mathcal{T}_1$ and $\mathcal{T}_2$ also cannot merge (be merged to) $\mathcal{T}_3$. This claim corresponds to the dominance of ``0'', implied by the three equations $0 \odot 0 = 0$, $1 \odot 0 = 0$, and $\varnothing \odot 0 = 0$ in \eqref{eq:rule}.
			\item If $\mathcal{T}_1$ and $\mathcal{T}_3$ do not share any common vertices, then merging $\mathcal{T}_1$ to $\mathcal{T}_2$ does not change the mergeability between $\mathcal{T}_2$ and $\mathcal{T}_3$. This statement corresponds to the relations $\varnothing \odot 0 = 0$, $\varnothing \odot 1 = 1$, and $\varnothing \odot \varnothing = \varnothing$ in \eqref{eq:rule}.
			\item If both $\mathcal{T}_1$ and $\mathcal{T}_2$ are mergeable to $\mathcal{T}_3$, then the union of $\mathcal{T}_1$ and $\mathcal{T}_2$ is still mergeable to $\mathcal{T}_3$. This statement is implied by the equation $1 \odot 1 = 1$ in \eqref{eq:rule}.
		\end{enumerate}
		Clearly, all the above statements correspond to the operators in \eqref{eq:rule}. 		
		Since the function $\mathscr{F} (\mathscr{M}, i, j)$  produces a reduced characterization matrix by the operations on the $i$-th and $j$-th rows as well as the $i$-th and $j$-th columns following the rules in \eqref{eq:operation}, the resulting characterization matrix indicates the mergeability of $\hat{\Pi}$, with $\mathcal{T}_i$ merged to $\mathcal{T}_j$ and the other pseudotrees untouched.
	\end{proof}

	\begin{exm} \label{ex1}
		Consider a directed simple graph with 10 vertices, as shown in Fig.~\ref{fig:mergeability}. Following Lemma~\ref{lem:treesmost}, the initial disjoint pseudotree covering $\Pi_0 = \{\mathcal{T}_1^{(0)}, \mathcal{T}_2^{(0)}, ..., \mathcal{T}_{9}^{(0)}\}$ in \eqref{eq:Pi0} is found, and each pseudotree has a single root vertex, which is not a sink and is labeled with the ordering number of the pseudotree. By the definition in \eqref{eq:Mergeability}, we construct the following matrix for characterizing the mergeability of $\Pi_0$.
		\begin{equation} \label{ex:M}
		\mathscr{M}^{(0)} = \begin{bmatrix}
		0   & 1   & \vn & \vn & 0   & \vn& 0   & \vn & \vn   \\%1
		0   & 0   & 1   & \vn & 0   & 0  & 0   & \vn & \vn   \\%2
		\vn & 1   & 0   & 0   &\vn  & 0  & \vn & 0   & 0     \\%3
		\vn & \vn & 0   & 0   &\vn  & 0  & \vn & 0   & 0     \\%4
		0   & 1   & \vn & \vn & 0   & 1  & 0   & 0   & \vn   \\%5
		\vn & 0   & 1   & 0   & 0   & 0  & 0   & 0   & 0     \\%6
		0   & 0   & \vn & \vn & 0   & 0  & 0   & 1   & \vn   \\%7
		\vn & \vn & 0   & 0   & 1   & 0  & 0   & 0   & 0     \\%8
		\vn & \vn & 0   & 0   & \vn & 0  & \vn & 0   & 0     \\%9
		\end{bmatrix}.
		\end{equation}
	Because $\mathscr{M}^{(0)}_{12} = 1$, the pseudotree $\mathcal{T}_{{1}}^{(0)}$ is mergeable to $\mathcal{T}_{{2}}^{(0)}$. 		
	The operation on the first two rows in $\mathscr{M}^{(0)}$ leads to
		\begin{align*}
		& \mathscr{M}^{(0)}_{1\star} \odot \mathscr{M}^{(0)}_{2\star} =
		\begin{bmatrix}
		0   & 0   & 1 & \vn   & 0 & 0 & 0   & \vn & \vn
		\end{bmatrix},
		\end{align*}
while the corresponding column operation provides
\begin{align*}
		&\mathscr{M}^{(0)}_{\star 1} \odot \mathscr{M}^{(0)}_{\star 2} =
		\begin{bmatrix}
		0   & 0   & 1 & \vn   & 0 & 0 & 0   & \vn & \vn
		\end{bmatrix}^\top.
		\end{align*}
%\red{IS THE EQUALITY OF THE RESULT OF THE ROW AND COLUMN OPERATION ACCIDENTAL OR STRUCTURAL?}\\

Next we replace the second row and column by the above products, and remove the first row and column of $\mathscr{M}^{(0)}$.
The reduction $\mathscr{M}^{(1)} = \mathscr{F}(\mathscr{M}^{(0)}, 1,2)$ then yields
		\begin{equation*}
		\mathscr{M}^{(1)} = \begin{bmatrix}
		0   & 1   & \vn & 0   & 0  & 0   & \vn & \vn   \\%2+1
		1   & 0   & 0   &\vn  & 0  & \vn & 0   & 0     \\%3
		\vn & 0   & 0   &\vn  & 0  & \vn & 0   & 0     \\%4
		0   & \vn & \vn & 0   & 1  & 0   & 0   & \vn   \\%5
		{0} & 1   & 0   & 0   & 0  & 0   & 0   & 0     \\%6
		0   & \vn & \vn & 0   & 0  & 0   & 1   & \vn   \\%7
		\vn & 0   & 0   & 1   & 0  & 0   & 0   & 0     \\%8
		\vn & 0   & 0   & \vn & 0  & \vn & 0   & 0     \\%9
		\end{bmatrix} \in \mathbb{M}^{8 \times 8},
		\end{equation*}
		which characterizes a new disjoint pseudotree covering: $\Pi_1 = \{\mathcal{T}_1^{(1)}, \mathcal{T}_2^{(1)}, ..., \mathcal{T}_{8}^{(1)}\}$, where $\mathcal{T}_1^{(1)} = \mathcal{T}_1^{(0)} \cup \mathcal{T}_2^{(0)}$ and $\mathcal{T}_i^{(1)} = \mathcal{T}_{i+1}^{(0)}$, for all $i = 2, 3, ..., 8$.
		
		\begin{figure}[!tp]\centering
			\centering
			\includegraphics[width=0.3\textwidth]{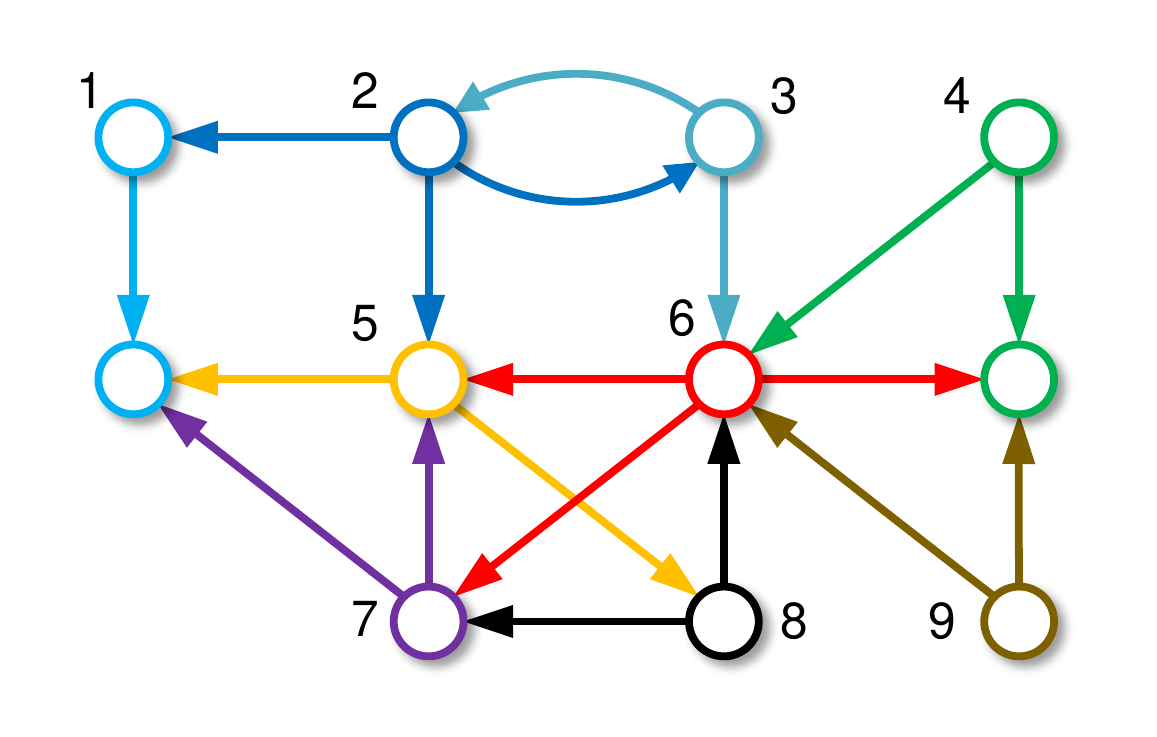}
			
			\caption{A directed simple graph with 11 vertices, which is decomposed into 9 disjoint pseudotrees, which are labeled with different colors.}
			\label{fig:mergeability}
		\end{figure}%
	\end{exm}
	
	The newly defined operation in \eqref{eq:operation} and the function \eqref{eq:func} allow us to represent the merging of two disjoint pseudotrees in a covering $\Pi$ by a reduction of its characteristic matrix $\mathscr{M}$.
	Based on this, we then proceed to a graph merging strategy that aims for a covering with the smallest possible number of disjoint pseudotrees. From the initial disjoint pseudotree covering $\Pi_0$, we obtain its characteristic matrix $\mathscr{M}^{(0)}$, according to which, we devise a heuristic algorithm to recursively integrate mergeable pseudotrees, see the description in Algorithm~\ref{alg:merge}.

	\begin{algorithm}[t]
		\caption{Disjoint pseudotree merging}
		\begin{algorithmic}[1]
			\REQUIRE  Extended graph $\widehat{\mathcal{G}}$ in Definition~\ref{def:extendednet}.
			
			\STATE  Initialize the disjoint pseudotree covering $\Pi_{0}$ as in \eqref{eq:Pi0}, with each pseudotree minimal.
			
			\STATE  Construct the characteristic matrix $\mathscr{M} =\mathscr{M}^{(0)}$ in \eqref{eq:2}.
			
			\REPEAT
			
			\STATE Find an entry $\mathscr{M}_{ij} = 1$, which is the only ``1'' entry in the ${i}$-th row of $\mathscr{M}$.
			
			\IF{there are multiple rows containing a single ``1'' entry}
			\STATE Let $i$ be the index of a row containing the most ``$\varnothing$'' entries.
			\ENDIF
			
			\STATE $\mathscr{M} \leftarrow \mathscr{F} (\mathscr{M}, i, j)$, and update $\Pi$ by merging the ${i}$-th pseudotree to the ${j}$-th one.
			
			\UNTIL each row of $\mathscr{M}$ contains more than one ``1'' entry.
			
			\REPEAT
			
			\STATE Find the $i$-th row of $\mathscr{M}$ with ``1'' entries and the most ``$\varnothing$'' entries.
			
			\STATE Select $\mathscr{M}_{ij} = 1$ as any ``1'' entry of the $i$-th row, and $\mathscr{M} \leftarrow \mathscr{F} (\mathscr{M}, i, j)$; update $\Pi$ by merging the ${i}$-th pseudotree to the ${j}$-th one.
			
			\UNTIL there is no ``1'' entry in $\mathscr{M}$.
			
			\RETURN $\Pi$.
			
		\end{algorithmic}
		\label{alg:merge}
	\end{algorithm}
	
	The scheme in Algorithm~\ref{alg:merge} is presented in two parts. In the first part, we find the row of the  characteristic  matrix with a unique ``1'' entry, as we aim to merge a pseudotree $\mathcal{T}_{{i}}$ to $\mathcal{T}_{{j}}$, if $\mathcal{T}_{{j}}$ is the only pseudotree that $\mathcal{T}_{{i}}$ is mergeable to. If there are multiple pairs that satisfy this condition, (e.g., in Fig.~\ref{fig:mergeability}, $\mathcal{T}_2$ is the only pseudotree that $\mathcal{T}_1$ and $\mathcal{T}_3$ can be merged to), we then merge $\mathcal{T}_i$ to $\mathcal{T}_j$, if $\mathcal{T}_i$ has more non-overlapped pseudotrees in $\Pi$, namely, the $i$-th row of $\mathscr{M}$ contains more ``$\varnothing$'' entries.  For instance, in Fig.~\ref{fig:mergeability}, as $\mathcal{T}_1$ has more non-overlapping pseudotrees, we merge $\mathcal{T}_1$ to $\mathcal{T}_2$ first. The reason behind this particular operation is that aggregating such a pair of pseudotrees would potentially cause less influence on the subsequent merging of the other pseudotrees in the covering.
	The second part of Algorithm~\ref{alg:merge} then deals with the remaining mergeable disjoint pseudotrees. Still, we tend to merge the pairs that have less overlaps with the other pseudotrees. When there does not exist any pair of mergeable pseudotrees, the merging procedure is finalized.

	\begin{exm}
		Consider the network in Fig.~\ref{fig:mergeability} and its initial disjoint pseudotree covering $\Pi_0 = \{\mathcal{T}_1^{(0)}, \mathcal{T}_2^{(0)}, ..., \mathcal{T}_{9}^{(0)}\}$, which is characterized by the matrix in \eqref{ex:M}.
		Following Algorithm~\ref{alg:merge},
		%    We continue to reduce the characteristic matrix using the same mechanism  and finally obtain
		%    \begin{equation} \label{ex:M4}
		%    \mathscr{M}_4 = \begin{bmatrix}
		%    0   & 0   & 0  & 0   & 0  \\%1+2
		%    0   & 0   & 0  & 0   & 0  \\%4
		%    0   & 1   & 0  & 0   & 1  \\%5+6
		%    0   & 0   & 0  & 0   & 0  \\%7+8
		%    0   & 0   & 0  & 0   & 0  \\%9
		%    \end{bmatrix} \in \mathbb{M}^{5 \times 5},
		%    \end{equation}
		%    which leads to the disjoint pseudotree covering $\Pi_4 = \{\mathcal{T}_1^4, \mathcal{T}_2^4, \mathcal{T}_3^4, \mathcal{T}_4^4, \mathcal{T}_{5}^4\}$, where $\mathcal{T}_1^4 = \mathcal{T}_1^0 \cup \mathcal{T}_2^0 \cup \mathcal{T}_3^0$, $\mathcal{T}_2^4 = \mathcal{T}_4^0$, $\mathcal{T}_3^4 = \mathcal{T}_5^0 \cup \mathcal{T}_6^0$, $\mathcal{T}_4^4 = \mathcal{T}_7^0 \cup \mathcal{T}_8^0$, and $\mathcal{T}_5^4 = \mathcal{T}_9^0$. The corresponding disjoint pseudotree covering $\Pi_4$ is illustrated in Fig.~\ref{fig:mergeresult}.
		the following operations are taken in order:
		$
		\mathscr{M}^{(1)} = \mathscr{F}(\mathscr{M}^{(0)}, 1, 2)$, $
		\mathscr{M}^{(2)} = \mathscr{F}(\mathscr{M}^{(1)}, 1, 2)$, $
		\mathscr{M}^{(3)} = \mathscr{F}(\mathscr{M}^{(2)}, 3, 4)$,
		and finally, we obtain
		%    \begin{equation}
		%        \mathscr{M}_5 = \mathscr{F}(\mathscr{M}_4, 3, 5) =
		%        \begin{bmatrix}
		%        0   & 0   & 0  & 0 \\%1+2
		%        0   & 0   & 0  & 0 \\%4
		%        0   & 0   & 0  & 0  \\%5+6+9
		%        0   & 0   & 0  & 0 \\%7+8
		%        \end{bmatrix},
		%    \end{equation}
		\begin{equation*}
		\mathscr{M}^{(4)} = \mathscr{F}(\mathscr{M}^{(3)}, 4, 5) = \begin{bmatrix}
		0   & 0   & 0  & 0   & 0     \\%2+1+3
		0   & 0   & 0  & 0   & 0     \\%4
		0   & 0   & 0  & 0   & 0     \\%6+5
		0   & 0   & 0  & 0   & 0     \\%8+7
		0   & 0   & 0  & 0   & 0     \\%9
		\end{bmatrix} \in \mathbb{M}^{5 \times 5}.
		\end{equation*}
		The corresponding disjoint pseudotree covering is given as $\hat{\Pi} = \{\mathcal{T}_1^{(4)}, \mathcal{T}_2^{(4)}, \mathcal{T}_3^{(4)}, \mathcal{T}_4^{(4)},\mathcal{T}_5^{(4)}\}$,
		with $\mathcal{T}_1^{(4)} = \mathcal{T}_1^{(0)} \cup \mathcal{T}_2^{(0)} \cup \mathcal{T}_3^{(0)}$, 
		$\mathcal{T}_2^4 = \mathcal{T}_4^{(0)}$, $\mathcal{T}_3^{(4)} = \mathcal{T}_5^{(0)} \cup \mathcal{T}_6^{(0)}$,
		 $\mathcal{T}_4^{(4)} = \mathcal{T}_7^{(0)} \cup \mathcal{T}_8^{(0)}$, and $\mathcal{T}_5^{(4)} = \mathcal{T}_9^{(0)}$. The resulting disjoint pseudotrees are depicted in Fig.~\ref{fig:mergeresult}, with their roots being labeled with numbers.
		
		\begin{figure}[!tp]\centering
			\centering
			\includegraphics[width=0.3\textwidth]{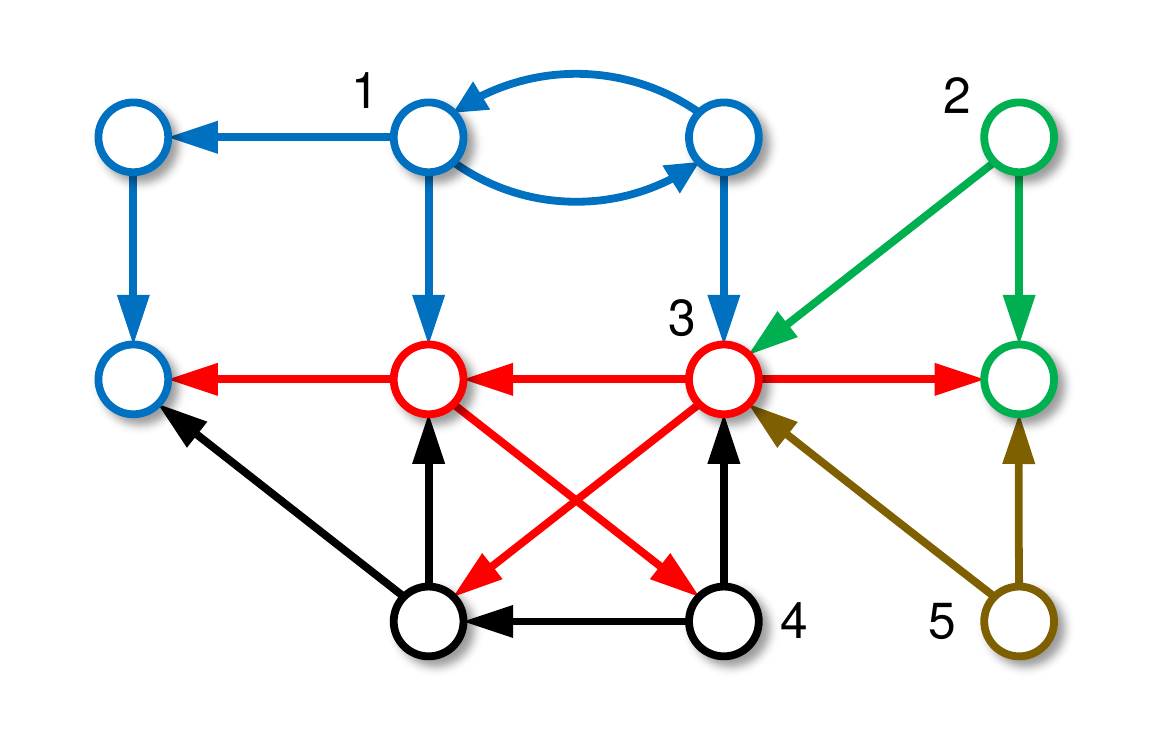}
			\caption{The resulting disjoint pseudotree covering of the directed graph in Fig.~\ref{fig:mergeability}, which is now partitioned into only 5 disjoint pseudotrees, labeled with different colors.}
			\label{fig:mergeresult}
		\end{figure}%
	\end{exm}

	\begin{rem}
	Algorithm~\ref{alg:merge} provides a heuristic but scalable procedure to find a local optimal solution in the sense that there will be no mergeable pseudotrees in the obtained covering. We may choose alternative heuristic merging procedures, e.g., a simple random merging, while the presented merging approach will potentially lead to a fewer number of pseudotrees.
	 It is possible to reach the exact minimum for dynamic networks of small size, for instance, the graphs in Fig.~\ref{fig:ex5decomp1} and Fig.~\ref{fig:mergeability}. However, for large-scale networks with e.g., up to hundreds or thousands of vertices, it is hard to guarantee the minimality in general. It is worth emphasizing that finding the minimal covering reflects as a new combinatorial optimization problem, whose optimal solution is not unique. Exploring the solution for this optimization problem itself requires a significant effort and  can lead to even new contributions to graph theory. Thus, it is  beyond the scope of this paper.
	\end{rem}
	
	\subsubsection{Allocation of Excitation Signals}
	For the synthesis problem of allocating excitation signals in a dynamic network for guaranteeing generic identifiability of the network model set, we apply Algorithm~\ref{alg:merge} to its extended graph $\widehat{\mathcal{G}}$ as a first step, aiming to decompose $\widehat{\mathcal{G}}$ into a minimal number of disjoint pseudotrees that cover all the parameterized edges of $\widehat{\mathcal{G}}$. Then, we proceed to the second step of our approach, which determines the locations of external excitation signals for the generic identifiability of $\mathcal{M}$. Specifically, in this step, we aim to solve the following problem: \textit{Given the extended graph $\widehat{\mathcal{G}}$ of a dynamic network model set $\mathcal{M}$, and let $\Pi$ be a disjoint pseudotree covering of $\widehat{\mathcal{G}}$, in which there do not exist mergeable pseudotrees. How to allocate the external excitation signals to make $\mathcal{M}$ generically identifiable}.

	To tackle the allocation problem, the process noises in the dynamic network have to be considered, which results in two facts: First, in the setting of the extended graph $\widehat{\mathcal{G}}$ in Definition~\ref{def:extendednet}, the vertices in the set $\widehat{V}$, which are also the roots of $|\widehat{V}|$ pseudotrees in $\Pi$, have been already excited by white noises in $e(t)$, or more precisely $e_{\theta}(t)$ in \eqref{eq:extg}. Second, it is also possible that one of the roots of a pseudotree in $\Pi$ has been excited by $e_0(t)$ in \eqref{eq:extg}, then it is not necessary to assign an excitation signal to a root of this pseudotree.  
	
	We thereby have a set $V_\mathrm{e} \subset V(\widehat{\mathcal{G}})$ with $|V_\mathrm{e}| = p$, in which the vertices are stimulated by white noises. More precisely, $V_\mathrm{e}$ includes $\widehat{V}$ and the vertices in $\mathcal{V}$ that are affected by $e_0(t)$. Define a set $\Pi_\mathrm{s} \subseteq \Pi$, which is generated by removing the elements in $\Pi$ that are  rooted at $V_\mathrm{e}$. Then, the following result is guaranteed by Corollary~\ref{coro:treecover}.
\begin{coro}\label{coro:Pis}
		Consider a set of vertices $\mathcal{R}: = \{\tau_1, \tau_2, ..., \tau_{|\Pi_\mathrm{s}|}\}$, where $\tau_i$ is a root of $\mathcal{T}_i \in \Pi_\mathrm{s}$. If  all the vertices in $\mathcal{R}$ are excited, then the dynamic network model set $\mathcal{M}$ is generically identifiable.
	\end{coro}
	
	Consequently, a direct strategy is to place an independent excitation signal to a root of each disjoint pseudotree in $\Pi_\mathrm{s}$. However, the condition Theorem~\ref{thm:treecover} allows us to further reduce the number of excitation signals.
	Thereby, we continue to check the necessity of each stimulated {vertex} in $\mathcal{R}$. If there exists a pseudotree $\mathcal{T}_k \in \Pi_\mathrm{s}$ such that each vertex in $V(\mathcal{T}_k)$ satisfies the vertex-disjoint condition $b_{\widehat{\mathcal{R}} \cup \widehat{V} \rightarrow \widehat{\mathcal{P}}_j} = |\widehat{\mathcal{P}}_j|$, where $\widehat{\mathcal{R}}: = \mathcal{R} \setminus \tau_k$, we then remove $\tau_k$ from $\mathcal{R}$. Simply put, if removing an element in $\mathcal{R}$ does not change the generic identifiability of the network model set $\mathcal{M}$, we can remove it. The detailed procedure is summarized in Algorithm~\ref{alg:assign}, which eliminates the removable elements in $\mathcal{R}$ iteratively.

	\begin{algorithm}[t]
		\caption{Allocation of excitation signals}
		\begin{algorithmic}[1]
			\REQUIRE The disjoint pseudotree covering $\Pi$ obtained from Algorithm~\ref{alg:merge}.
			
			\STATE $\Pi_\mathrm{s} \leftarrow \Pi$
			\FOR{$\mathcal{T}_k \in \Pi$, $k = 1: |\Pi|$}
			\IF{$\Upsilon(\mathcal{T}_k) \cap V_\mathrm{e} \ne \emptyset$} \STATE $\Pi_\mathrm{s} \leftarrow \Pi_\mathrm{s} \setminus \mathcal{T}_k$.
			\ENDIF
			\ENDFOR
			\STATE Let $\mathcal{R}: = \{\tau_1, \tau_2, ..., \tau_{|\Pi_\mathrm{s}|}\}$, with $\tau_i$ a root of $\mathcal{T}_i \in \Pi_\mathrm{s}$
			
			\FOR{$\mathcal{T}_k \in \Pi_\mathrm{s}$, $k = 1: |\Pi_\mathrm{s}|$}
			\STATE $\widehat{\mathcal{R}} \leftarrow \mathcal{R} \setminus \tau_k$
			\IF{$b_{\widehat{\mathcal{R}} \cup \widehat{V} \rightarrow \widehat{\mathcal{P}}_j} = |\widehat{\mathcal{P}}_j|$, $\forall~j \in V(\mathcal{T}_k)$}
			\STATE ${\mathcal{R}} \leftarrow \mathcal{R} \setminus \tau_k$.
			\ENDIF
			\ENDFOR
			
			\RETURN $\mathcal{R}$.
		\end{algorithmic}
		\label{alg:assign}
	\end{algorithm}
	
	\begin{exm}
		Continue the network example in Fig.~\ref{fig:mergeresult}, which depicts a disjoint pseudotree covering resulting from Algorithm~\ref{alg:merge}. Suppose that the roots of the pseudotrees 2 and 5 are excited by white noises {in $e$}. Then, through Algorithm~\ref{alg:assign}, we do not need to excite the root of the pseudotree 3. Thus, only two {additional} excitation signals {in $r$} are required to achieve generic identifiability, and one of the possible allocations is illustrated in Fig.~\ref{fig:assignresult}. Note that in $\widehat{\mathcal G}$, there are two sources, and the maximal in-degree is 4. Thus, it follows from \eqref{eq:bound} that $K$ is lower-bounded by  $\max\left\{|\mathcal{S}_\mathrm{ou}(\widehat{\mathcal{G}})|, \max_{j \in V(\widehat{\mathcal{G}})} {|\widehat{\mathcal{P}}_j|}\right\}
		- p = 2$, which means that 2 is the minimal number of excitation signals {in $r$} that are needed for generic identifiability.
		\begin{figure}[!tp]\centering
			\centering
			\includegraphics[width=0.3\textwidth]{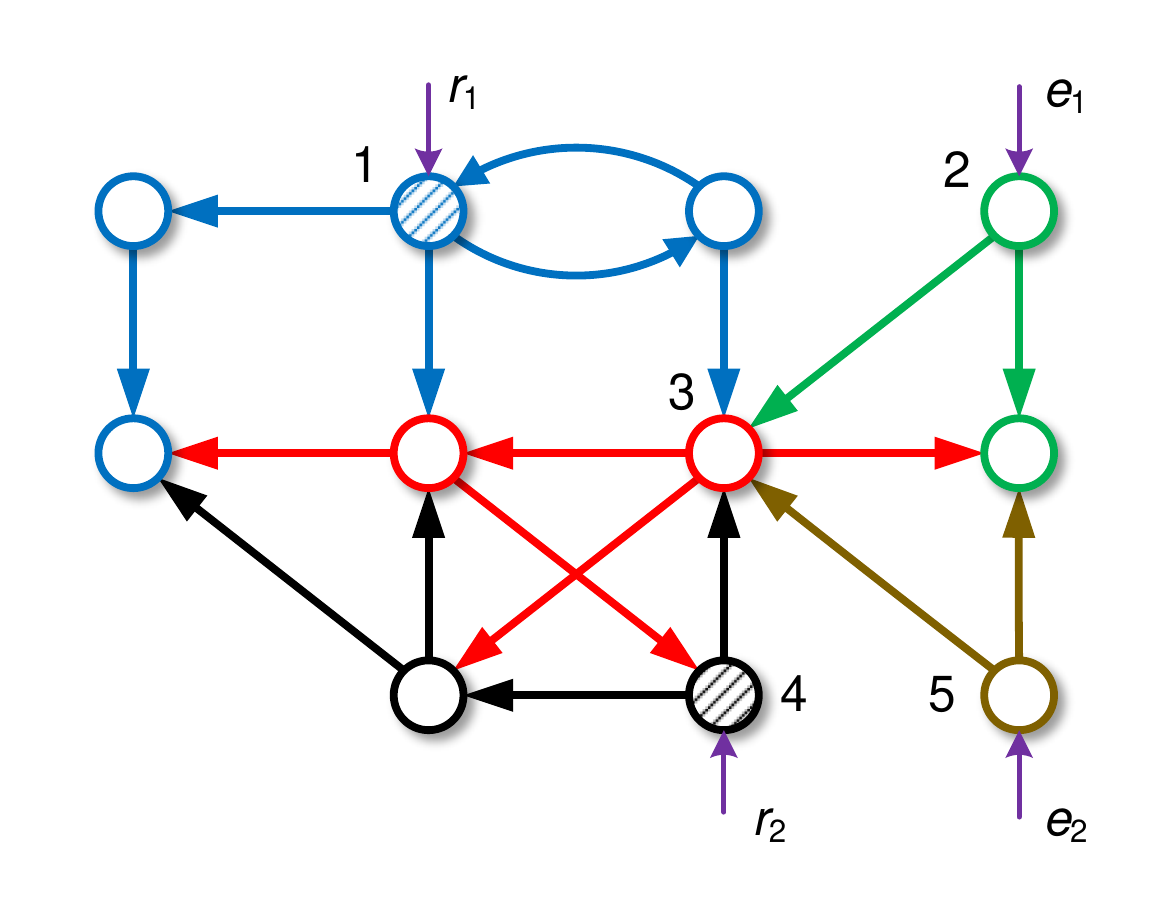}
			
			\caption{One of the solutions for allocating the excitation signals is to assign the shadowed vertices as the excited ones.}
			\label{fig:assignresult}
		\end{figure}%
	\end{exm}

\begin{rem}
Because of Assumption \ref{ass3} all non-zero modules in $G$ need to be parameterized in order for the graph-based result to be applicable. However also in case of non-parameterized, known modules in $G$ that are unequal zero, the results apply as long as the known modules in $G$ are chosen generic values, i.e. they do not introduce any dependence relations. In this situation, the pseudotree covering results presented in this section remain to hold, but require that only the parameterized modules in $G$ and $H$ need to be covered by pseudotrees. A further analysis of this situation is beyond the scope of the current paper.
\end{rem}	
	
	%%%%%%%%%%%%%%%%%%%%%%%%%%%%%%%%%%%%%%%%%%%%%%%%%%%%%%%%%%%%%%%%%%%%%%%%%%%%%%%%
	\section{A Dual Problem: Selecting Measured Vertices}
	\label{sec:dual}
	In the previous sections, we have considered the situation that all the vertex signals are measured, while only partial vertices are selected to be excited. 
%	\red{\st{Based on the relation between the generic identifiability of a dynamic network model set and the disjoint pseudotree covering of its extended graph, a graph merging strategy is proposed to find a minimal number of disjoint pseudotrees covering all the edges of the extended graph, and then excitation signals can be allocated according to this covering.}}
%	
	The works in e.g., \cite{hendrickx2018identifiability,henk2018identifiability} consider a dual model setting, in which all the vertices are stimulated by independent excitation sources, but only a subset of vertex signals are measured. In this section, we show that our approach can be also adapted to solve the dual problem in this setting, that is to select a minimal number of measured vertices for generic identifiability. Specifically, this section considers a network with the following dynamics
	\begin{equation} \label{eq:net2}
	\begin{split}
	w(t) &= G(q) w(t) + r(t) + v(t), \\
	y(t) &= C w(t),
	\end{split}
	\end{equation}
	where $w(t)$,  $r(t)$, and $v(t)$ are vertex signals, excitation signals and process noises defined in \eqref{eq:net}. The measurement signal $y(t) \in \mathbb{R}^{m}$ is a vector
	consisting of selected internal variables in the network \eqref{eq:net2}, and $C \in \mathbb{R}^{m \times L}$ is a binary matrix with $C_{ij} = 1$ if $y_i(t) = w_j(t)$, and $C_{ij} = 0$ otherwise. For ease of exposition, we will consider the situation that $v \equiv 0$. Define
	\begin{equation} \label{eq:modset2}
	\check{\mathcal{M}}: = \{G(q,\theta), \theta \in \Theta\}
	\end{equation}
	as the network model set associating with dynamic networks in form of \eqref{eq:net2}, where all the nonzero entries in $G(q, \theta)$ are parameterized.
	We are interested in the question: \textit{How to select a minimal number of measurement signals $y(t)$ such that $\check{\mathcal{M}}$ is generically identifiable, i.e.,
	almost all network modules $G_{ji}$ can be uniquely
	identified from $C (I - G)^{-1}$}.
	
	Following  \cite{hendrickx2018identifiability,bazanella2017identifiability}, a path-based condition for the generic identifiability of $\check{\mathcal{M}}$ is that the maximum number of mutually vertex-disjoint
	paths from $\mathcal{N}_j^+$ to $\mathcal{C}$ is  equal to $| \mathcal{N}_j^+|$
	for all $i \in V(\check{\mathcal{G}})$, where $\mathcal{N}_j^+$ is the set of the out-neighbors of $j$.
%	$
%	{\mathcal{P}}_j^+: = \{i \in \mathcal{N}_j^+ \mid G_{ij}(\theta)~\text{is parameterized}\}.
%	$
	
	Thereby, we define the concept of \textit{anti-pseudotrees}. A simple connected graph $\check{\mathcal{T}}$ is an anti-pseudotree if $|\mathcal{N}_i^+| \leq 1$, for all $i \in V(\check{\mathcal{T}})$. An anti-pseudotree can be generated by reversing the orientations of all the edges of a pseudotree in Definition~\ref{def:pseudotree}. Furthermore, $\Upsilon(\check{\mathcal{T}})$ is a set of roots of an anti-pseudotree $\check{\mathcal{T}}$ such that each vertex in $\check{\mathcal{T}}$ has a unique directed path toward all the vertices in $\Upsilon(\check{\mathcal{T}})$. Two anti-pseudotrees are disjoint if they do not share any common edges, and all the edges incident to each vertex are included in the same anti-pseudotree. Analogously, we can characterize the generic identifiability of a dynamic network model set $\check{\mathcal{M}}$ using disjoint anti-pseudotrees.
	\begin{pro}
		Consider a network model set $\check{\mathcal{M}}$ composed of network models described in \eqref{eq:net2}.
		Let $\mathcal{Y}: = \{y_1, y_2, ..., y_m\}$ be the set of measured vertices. The network model set $\check{\mathcal{M}}$ is generically identifiable if and only if
		one of the following conditions hold:
		\begin{enumerate}
			\item There exists a set of disjoint anti-pseudotrees, $\check\Pi = \{\check{\mathcal{T}}_1, \check{\mathcal{T}}_2, ..., \check{\mathcal{T}}_n\}$ with $n \leq m$,
			such that each anti-pseudotree has at least one root vertex being measured. 
			namely, $\Upsilon(\check{\mathcal{T}}_k) \cap \mathcal{Y} \ne \emptyset$, $\forall~k \in \{1, 2, ..., n\}$;
			\item There exists a set of disjoint anti-pseudotrees, $\check\Pi = \{\check{\mathcal{T}}_1, \check{\mathcal{T}}_2, ..., \check{\mathcal{T}}_n\}$ with $n > m$, such that $y_k \in \Upsilon(\check{\mathcal{T}}_k)$, $\forall~k \in \{1, 2, ..., m\}$
			and $b_{{\mathcal{N}}_j^+ \rightarrow \mathcal{\mathcal{Y}}} = |\widehat{\mathcal{N}}_j^+|$, $\forall~j \in V(\check{\mathcal{T}}_{m+1})\cup \cdots \cup V(\check{\mathcal{T}}_{n})$.
		\end{enumerate}
	\end{pro}
	The proof follows a similar reasoning as the proof of Theorem~\ref{thm:treecover} and Corollary~\ref{coro:treecover}, thus it is omitted here. Moreover, the minimal number of measurement signals that guarantees generic identifiability is bounded as
	\begin{equation*}
	\max\left\{|\mathcal{S}_\mathrm{in}({\mathcal{G}})|, \max_{j \in V({\mathcal{G}})} {|{\mathcal{N}}_j^+|}\right\}
	\leq m
	\leq \check{\kappa}({\mathcal{G}}),
	\end{equation*}
	where $\mathcal{G}$ is the underlying graph of the network \eqref{eq:net2}, and $\check{\kappa}({\mathcal{G}})$ is the minimal number of disjoint anti-pseudotrees that cover all the parameterized edges in $\mathcal{G}$.
	
	Analogously, we can devise a similar algorithm as Algorithm~\ref{alg:merge} to find the minimal covering and then remove unnecessary measurements as Algorithm~\ref{alg:assign} such that a set of measured vertices are selected. Consider an example shown in Fig.~\ref{fig:anti}, which is taken from \cite{hendrickx2018identifiability}. The network in this example can be decomposed into 4 disjoint anti-pseudotrees. Our approach then suggests  taking the measurements from the roots of these anti-pseudotrees. Consequently, generic identifiability can be achieved with 4 measured vertices.
	
	\begin{figure}[!tp]\centering
		\centering
		\includegraphics[width=0.33\textwidth]{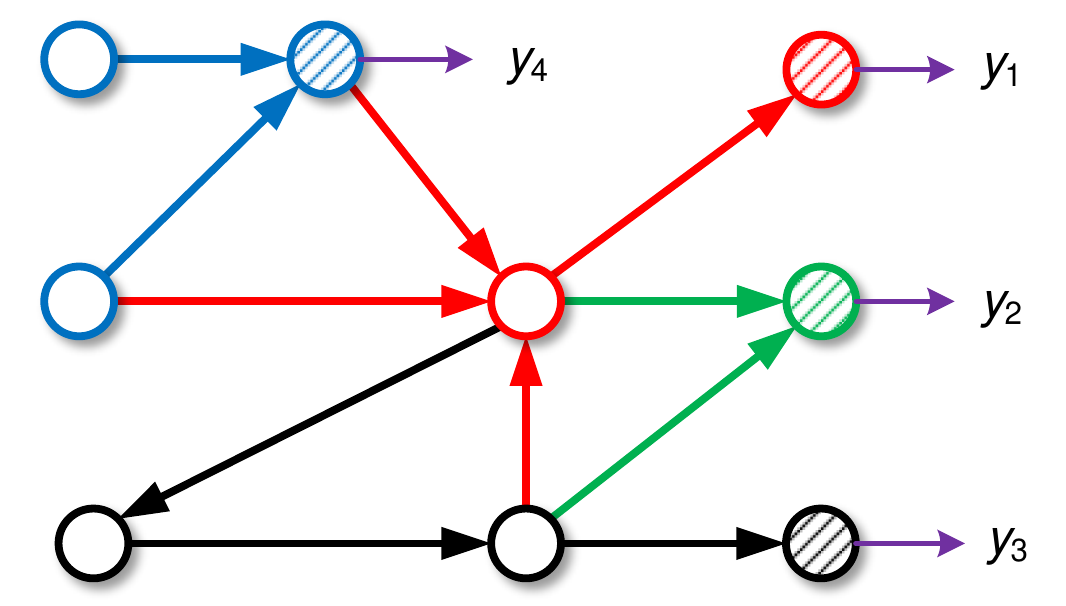}
		
		\caption{A network can be partitioned into 4 anti-pseudotrees highlighted by different colors. We select the shadowed vertices as measurement signals to achieve generic identifiability.}
		\label{fig:anti}
	\end{figure}%
	
	%%%%%%%%%%%%%%%%%%%%%%%%%%%%%%%%%%%%%%%%%%%%%%%%%%%%%%%%%%%%%%%%%%%%%%%%%%%%%%%%
	\section{Conclusion}
	\label{sec:conclusion}
	
	In this paper, we have addressed an excitation allocation problem for dynamic networks. Considering correlations
	between disturbances and non-parameterized modules to be present in a network model set, the goal is to select a minimal number of external excitation signals such that the model set becomes generically identifiable from measurement data. This provides conditions for the consistent identification of all parameterized modules in the model set.
	To this end, the notion of pseudotrees is introduced, and a novel necessary and sufficient graph-theoretic condition  has been provided based on disjoint pseudotrees to characterize the generic identifiability of a dynamic network model set. Based on this condition, an effective approach has been proposed, aiming to find a minimal number of excitation signals and their locations, where the number of the excitations is upper-bounded by the minimal number of disjoint pseudotrees that cover all the edges of the extended graph, and the locations of the excitations can be potentially selected as the roots of these pseudotrees.
	For future work, the identifiability problem in a dynamic network with partial measured and partial excited vertices is of interest. Specifically, it is worth investigating the research question how to place excitation signals in a network to achieve identifiability in the case that only partial measurements are available.
	
	%%%%%%%%%%%%%%%%%%%%%%%%%%%%%%%%%%%%%%%%%%%%%%%%%%%%%%%%%%%%%%%%%%%%%%%%%%%%%%
	%\appendices
	%\section{Proof of Theorem \ref{thm:controlrank}}
	%
	%
	%$\hfill{} \Box$
	
	%\section{}
	
	%\section*{ACKNOWLEDGMENT}
	%
	%The preferred spelling of the word ÒacknowledgmentÓ in America is without an ÒeÓ after the ÒgÓ. Avoid the stilted expression, ÒOne of us (R. B. G.) thanks . . .Ó  Instead, try ÒR. B. G. thanksÓ. Put sponsor acknowledgments in the unnumbered footnote on the first page.
	
	\bibliographystyle{IEEEtran}
	\bibliography{netid}

\end{document}